\documentclass[11pt,reqno]{amsart}
\usepackage[T1]{fontenc}
\usepackage{amsfonts}
\usepackage{amssymb}
\usepackage{mathrsfs}
\usepackage[latin1]{inputenc}
\usepackage{amsmath}
\usepackage[english]{babel}

\newtheorem{theorem}{Theorem}[section]
\newtheorem{lemma}[theorem]{Lemma}

\newtheorem{remark}[theorem]{Remark}
\newtheorem{prop}[theorem]{Proposition}
\newtheorem{example}[theorem]{Example}
\newtheorem{corollary}[theorem]{Corollary}

\numberwithin{equation}{section}

\newcommand{\R}{{\mathbb R}}

\newcommand{\C}{{\mathbb C}}
\newcommand{\N}{{\mathbb N}}

\newcommand{\cL}{{\mathcal L}}

\newcommand{\cH}{{\mathcal H}}

\newcommand{\cB}{{\mathcal B}}

\newcommand{\cV}{{\mathcal V}}

\newcommand{\Si}{\Sigma}

\newcommand{\ve}{\varepsilon}

\newcommand{\al}{\alpha}

\newcommand{\g}{\gamma}

\newcommand{\vp}{\varphi}
\newcommand{\si}{\sigma}
\newcommand{\su}{\subseteq}

\newcommand{\la}{\langle}
\newcommand{\ra}{\rangle}
\newcommand{\wt}{\widetilde}
\newcommand{\s}{\setminus}
\newcommand{\sC}{\mathsf{C}}
\newcommand{\ov}{\overline}

\newcommand\proj{\mathop{\rm proj\,}}
\newcommand\ind{\mathop{\rm ind\,}}
\newcommand\Ker{\mathop{\rm Ker}}

\begin{document}

\title{The Ces\`aro operator on duals of power series   spaces of infinite type}

\author{Angela A. Albanese, Jos\'e Bonet, Werner \,J. Ricker}

\thanks{This article is accepted for publication in Journal of Operator Theory}

\address{ Angela A. Albanese\\
Dipartimento di Matematica e Fisica ``E. De Giorgi''\\
Universit\`a del Salento- C.P.193\\
I-73100 Lecce, Italy}
\email{angela.albanese@unisalento.it}

\address{Jos\'e Bonet \\
Instituto Universitario de Matem\'{a}tica Pura y Aplicada
IUMPA \\
Universitat Polit\`ecnica de
Val\`encia\\
E-46071 Valencia, Spain} \email{jbonet@mat.upv.es}

\address{Werner J.  Ricker \\
Math.-Geogr. Fakult\"{a}t \\
 Katholische Universit\"{a}t
Eichst\"att-Ingol\-stadt \\
D-85072 Eichst\"att, Germany}
\email{werner.ricker@ku-eichstaett.de}
\markboth{A.\,A. Albanese, J. Bonet, W. \,J. Ricker}%
{\MakeUppercase{ }}

\begin{abstract}
A detailed  investigation is made of the continuity, spectrum and mean ergodic properties of the Ces\`aro operator $\sC$ when acting
on the strong  duals of power series  spaces of infinite type. There is a dramatic difference in the nature of the spectrum of $\sC$
depending on whether or not the strong dual space (which is always Schwartz) is nuclear.
\end{abstract}

\maketitle

\markboth{A.\,A. Albanese, J. Bonet and W.\,J. Ricker}%
{\MakeUppercase{The Ces\`aro operator }}

\section{Introduction and Notation.}

The discrete Ces\`aro operator $\sC$ is defined on the linear space $\C^\N$ (consisting of all scalar sequences) by
\begin{equation}\label{eq.Ce-op}
\sC x:=\left(x_1, \frac{x_1+x_2}{2}, \ldots, \frac{x_1+\ldots +x_n}{n}, \ldots\right),\quad x=(x_n)_{n\in\N}\in \C^\N.
\end{equation}
The linear operator $\sC$ is said to \textit{act} in a vector subspace  $X\su\C^\N$ if it maps $X$ into itself. Of particular interest is the situation when $X$ is a Fr\'echet space
 or an (LF)-space.      Two fundamental questions in this case are: Is $\sC\colon X\to X$ continuous and, if so, what is its spectrum?
For a large collection of classical Banach spaces $X\su \C^\N$ where precise answers are known we refer to the Introductions in \cite{ABR-9}, \cite{ASM}, for
example.
The discrete Ces\`aro operator $\sC$ acting on the Fr\'echet sequence  space $\C^{\N}$,  on  $\ell^{p+} := \cap_{q>p} \ell_q$, and on the
power series spaces $\Lambda_0(\al):= \Lambda_0^1 (\al)$ of finite type   was investigated in \cite{ABR-7},
 \cite{ABR-1}, \cite{ASM}, respectively. The aim of this paper is to investigate the behaviour of $\sC$ when it acts on the \textit{strong  duals}
 $(\Lambda^1_{\infty}(\alpha))'$  of power series   spaces  $\Lambda^1_{\infty}(\alpha)$  of \textit{infinite type}.
  Power series spaces of infinite type play an important role in the isomorphic classification of Fr\'echet spaces,  \cite{MV}, \cite{V1}, \cite{V2}.
The reason for concentrating on the infinite type dual spaces $(\Lambda^1_\infty (\al))'$ is that the Cesàro operator $\sC$ \textit{fails} to be
continuous on ``most'' of the finite type dual spaces $(\Lambda_0^1 (\al))'$. This is explained more precisely in an Appendix (Section 5) at the end of the paper.

In order to describe the main results we require some notation and definitions.

Let $X$ be a locally convex Hausdorff space (briefly, lcHs) and $\Gamma_X$ a system of continuous seminorms determining the topology of $X$.
Let $X'$ denote the space of all continuous linear functionals on $X$.
 The family of all bounded subsets of $X$ is denoted by $\cB (X)$.
Denote the identity operator on $X$ by $I$. Let  $\cL(X)$ denote the space of all continuous linear operators from $X$ into itself.
For  $T\in \cL(X)$, the \textit{resolvent set} $\rho(T)$ of $T$ consists of all $\lambda\in\C$ such that $R(\lambda,T):=(\lambda I- T)^{-1}$ exists in $\cL(X)$. The set  $\sigma(T):=\C\setminus \rho(T)$ is called the \textit{spectrum} of $T$.
The \textit{point spectrum} $\sigma_{pt}(T)$ of $T$ consists of all $\lambda\in\C$ such that $(\lambda I-T)$ is not injective. If we need to stress the space $X$, then we also write $\sigma(T;X)$, $\sigma_{pt}(T;X)$ and $\rho(T;X)$. Given $\lambda, \mu\in \rho(T)$  the \textit{resolvent identity} $R(\lambda,T)-R(\mu,T)=(\mu-\lambda) R(\lambda,T)R(\mu,T)$  holds.
Unlike for Banach spaces, it may happen that $\rho(T)=\emptyset$ (cf.\,Remark \ref{r26}(ii)) or that $\rho(T)$ is not open in $\C$; see Proposition \ref{P.spectrum}(i) for example.
That is why some authors prefer the subset $\rho^*(T)$ of $\rho(T)$
 consisting of all $\lambda\in\C$ for which  there exists $\delta>0$ such that  the open disc  $ B(\lambda,\delta):=\{z\in\C\colon |z -\lambda| < \delta\} \su \rho(T)$
  and  $\{R(\mu,T) \colon \ \mu\in  B(\lambda,\delta)  \}$ is equicontinuous in $\cL(X)$. If $X$ is a Fr\'echet space or even an  (LF)-space, then  it suffices that such sets are
   bounded in $\cL_s(X)$, where $\cL_s(X)$ denotes  $\cL(X)$ endowed with the strong operator topology $\tau_s$ which is determined by the  seminorms
  $ T \mapsto q_x(T):=q(Tx)$, for  all  $x\in X$ and $q\in \Gamma_X$. The advantage of $\rho^*(T)$, whenever it is non-empty, is that it is open and the resolvent map
  $R\colon \lambda\mapsto R(\lambda,T)$ is holomorphic from $\rho^*(T)$ into $\cL_b(X)$, \cite[Proposition 3.4]{ABR-5}. Here $\cL_b(X)$ denotes  $\cL(X)$ endowed with the
     lcH-topology $\tau_b$ of uniform convergence on members of $\cB (X)$; it is   determined by the  seminorms $T \mapsto q_B(T):=\sup_{x\in B}q(Tx)$, for  $T\in \cL(X)$,
  for all  $B\in \cB(X)$ and $q\in \Gamma_X$. Define $\sigma^*(T):=\C\setminus \rho^*(T)$, which  is a closed set  containing $\sigma(T)$.  If $T\in \cL(X)$ with $X$  a Banach space,
  then $\sigma(T)=\sigma^*(T)$. In \cite[Remark 3.5(vi), p.265]{ABR-5}  an example of a continuous linear operator $T$ on a Fr\'echet space $X$   is presented such that
  $\ov{\sigma(T)}\subset\sigma^*(T)$ properly. For undefined concepts concerning lcHs' see \cite{Ja}, \cite{MV}.

Each positive, strictly increasing sequence $\al = (\al_n)$ which tends to infinity generates a power series  space $\Lambda^1_{\infty}(\alpha)$ of
infinite type; see Section 2. The strong dual $E_\al \su \C^\N $ of $\Lambda^1_{\infty}(\alpha)$ is then a co-echelon space, i.e., a particular kind of inductive limit
of Banach spaces (of sequences), which is necessarily a Schwartz space in our setting. It turns out (cf. Proposition \ref{C-Cont}) that always $\sC \in \cL (E_\al)$.
Furthermore, it is known that the nuclearity of the space $E_\al$ is characterized by the condition $\sup_{n \in \N } \frac{\log (n)}{\al_n} < \infty $. Remarkably,
this is equivalent to the operator $\sC \in \cL (E_\al)$ being invertible, i.e., $0 \in \rho (\sC ; E_\al)$; see Proposition \ref{P.sp0}. Actually, the main results of this
section (namely, Proposition \ref{P.spectrum} and Corollary \ref{C210})  establish the equivalence of the following assertions:
\begin{enumerate}
  \item [(i)]
$E_\al $  \textit{ is nuclear}.
  \item [(ii)]
 $\si  (\sC ; E_\al) = \si_{pt}  (\sC ; E_\al)$ .
  \item [(iii)]
$\si  (\sC ; E_\al) = \{\frac 1 n : n \in \N \}$.
\end{enumerate}
Moreover, in this case we have $\si^*  (\sC ; E_\al)  = \{0\} \cup \si  (\sC ; E_\al) $. So, whenever $E_\al$ is nuclear, the spectra $\si_{pt}  (\sC ; E_\al) $,
$\si  (\sC ; E_\al) $ and $\si^*  (\sC ; E_\al) $ are completely identified. In particular, these spectra of $\sC$ are independent of $\al$.

The operator $ D \in \cL (\C^\N)$ of differentiation (defined in the obvious way) is closely connected to the Cesàro operator $\sC \in \cL (\C^\N)$ via the
identity (valid in $\cL (\C^\N )$)
$$
    \sC^{-1} = (I- S_r) D S_r ,
$$
where $S_r \in \cL (\C^\N)$ is the right-shift operator. It is always the case that $S_r \in \cL (E_\al)$ whenever $\al_n \uparrow \infty $. Moreover, it
follows from (i)-(iii) above that $\sC^{-1} \in \cL (E_\al)$ precisely when $E_\al$ is nuclear. So, the above identity for $\sC^{-1}$ suggests that there should be a connection
between the continuity of $D $ on $E_\al$ and the nuclearity  of $E_\al$. This is clarified by Proposition \ref{pr25}. Namely, $D$ is
continuous on $E_\al$ if and only if $E_\al$ is both nuclear and $\sup_{n \in \N } \frac{\al_{n+1}}{\al_n} < \infty $. Remark \ref{r26}(i) shows that these
two conditions are independent of one another.

Section 3 identifies the spectra of $\sC \in \cL (E_\al)$ in the case when $E_\al$ is not nuclear. We have seen if $E_\al$ is nuclear, then
$\si  (\sC ; E_\al) $ is a bounded, infinite and countable set with no accumulation points. For $E_\al$ \textit{non-nuclear} the spectrum of $\sC$
 is very different. Indeed, in this case
 $$
\si  (\sC ; E_\al) = \{0,1\} \cup \{\lambda \in \C : |\lambda - \textstyle \frac 1 2 | < \frac 1 2 \} \mbox{ and } \si^*  (\sC ; E_\al) =
\{\lambda \in \C : |\lambda - \textstyle \frac 1 2 | \le \frac 1 2 \}
 $$
 whenever $\sup_{n \in \N } \frac{\log (\log (n))}{\al_n} < \infty $, whereas
$$
\si  (\sC ; E_\al) = \si^*  (\sC ; E_\al) =   \{\lambda \in \C : |\lambda - \textstyle \frac 1 2 | \le \frac 1 2 \}
$$
otherwise; see Proposition \ref{pr34}. Again the spectra of $\sC$ are independent of $\al$.

J. von Neumann (1931) proved that unitary operators $T$ in Hilbert space are mean ergodic, i.e., the sequence of its averages $\frac 1 n
\sum^n_{m=1} T^m$, for $n \in \N $, converges for the strong operator topology (to a projection). Ever since, intensive research has been
undertaken to identify the mean ergodicity of individual (and classes) of operators both in Banach spaces and non-normable lcHs'; see
\cite{ABR-0}, \cite{K} for example, and the references therein. In Section 4 it is shown, for every sequence $\al $ with $\al_n \uparrow \infty $, that the Cesàro operator
$\sC \in \cL (E_\al)$ is always  power bounded, (uniformly) mean ergodic and $E_\al = \Ker (I-\sC) \oplus \ov{(I-\sC )(E_\al)} $; see Proposition \ref{p.uniforE}.
Actually, even the sequence $\{\sC^m\}^\infty_{m=1}$ of the iterates of $\sC$ (not just its averages) turns out to be convergent, not only in $\cL_s (E_\al)$  but also in
$\cL_b (E_\al)$; see Proposition \ref{p.iterE}. Furthermore, if $E_\al $ is nuclear, then   the range $(I- \sC)^m (E_\al)$ of the operator $(I- \sC)^m$ is a closed subspace of
$E_\al$ for each $m \in \N $ (cf. Proposition \ref{p.rangeE}). For $m=1$ this is an analogue, for the operator $\sC \in \cL (E_\al)$, of a result of
M. Lin for arbitrary uniformly mean ergodic   Banach space operators $T$ which satisfy $\lim_{n \to \infty } \frac{\|T^n\|}{n} =0$, \cite{Li}.

\section{The  Spectrum of $\sC$ in the nuclear case}

Let $\alpha:=(\alpha_n)$ be a  positive, strictly increasing sequence tending to infinity, briefly, $\alpha_n \uparrow \infty $. Let $(s_k)\su (1,\infty)$
be another strictly increasing  sequence satisfying  $s_k \uparrow \infty $ . For each $k \in\N$,  define $v_k : \N \to (0, \infty )$ by
$v_k(n):=  s_k^{-\alpha_n}$ for $n \in \N $. Then $v_k(n)\geq v_k(n+1)$, for $n  \in \N $, i.e., $v_k $ is a decreasing sequence,
and $v_k \geq v_{k+1}$ pointwise on $\N $  for all $k  \in\N$. Set $\cV:=(v_k)$ and note that $v_k \in c_0 $ for all $k \in \N $.

Define the co-echelon spaces $E_\alpha:=\ind_k c_0(v_k)$, that is, $E_\alpha$ is the (increasing) union of the  weighted Banach spaces $c_0(v_k)$, $k\in\N$, endowed with the finest lcH-topology such that each natural  inclusion map $c_0(v_k)\hookrightarrow E_\alpha$ is continuous. Since $\lim_{n\to\infty}\frac{v_{k+1}(n)}{v_k(n)}=0$,  for  $k\in\N$, implies that
$\ell_\infty(v_k) \su c_0 (v_{k+1})$ continuously,  for $k \in \N $, it follows that also $E_\alpha:=\ind_k  \ell_\infty (v_k) $. Observing that the power series
space $\Lambda^1_\infty(\alpha):=\proj_k\ell_1(v_k^{-1})$ of \textit{infinite type} is Fr\'echet-Schwartz (hence, distinguished), \cite[p.\,357]{MV}, it follows
that $E_\alpha := \ind_k c_0 (v_k) = \ind_k \ell_\infty (v_k) = (\Lambda^1_\infty (\al))'$ is the \textit{strong dual}  of $\Lambda^1_\infty (\al)$,
\cite[Remark 25.13]{MV}. The condition  $\frac{v_{k+1}}{v_k} \in c_0$ for $k \in \N $ implies  that
  $E_\alpha$ is always  a (DFS)-space, \cite[p.\,304]{MV},  and in particular, a Montel space, \cite[Remark 24.24]{MV}.
Note that power series spaces in \cite[Chapter 24]{MV} are defined  using $\ell_2$-norms.
It follows from \cite[Proposition 29.6]{MV} that  $\Lambda^1_\infty(\alpha)$ is a nuclear Fr\'echet space (equivalently, $E_\alpha$ is a (DFN)-space) if and only if $\sup_{n\in\N}\frac{\log n}{\alpha_n}<\infty$. This criterion plays a relevant role throughout this section. As the space $E_\alpha$ does not change if $(s_k)$ is replaced by any other strictly increasing sequence
 in $ (1, \infty )$ tending to infinity, we  sometimes choose $s_k :=e^k$, $k\in\N$.  For each $k \in \N $, define the norm
$$
  q_k(x):=\sup_{n\in\N}v_k(n)|x_n|, \qquad  x=(x_n) \in \ell_\infty (v_k) ,
$$
 whose restriction to $c_0 (v_k)$ is the norm of $c_0 (v_k)$. Observe, for each $k \in \N $, that $c_0 (v_k)  \su c_0 (v_l)$
for every $l \in \N $ with $l \ge k $, and
\begin{equation}\label{21}
    q_l (x) \le q_k (x) , \qquad x \in c_0 (v_k)  .
\end{equation}

As general references for co-echelon spaces we refer to \cite{BMS}, \cite{B}, \cite{Ko}, \cite{MV}, for example.

\begin{prop}\label{C-Cont}
For each $ \al_n \uparrow \infty $  the Ces\`aro operator satsifies  $\sC \in \cL (E_\alpha)$.
\end{prop}

\begin{proof}
Since each sequence  $v_k$ , for $k \in \N $, is  decreasing,   Corollary 2.3(i) of \cite{ABR-9} implies that the Cesàro operator at
each step, namely  $\sC\colon c_0(v_k)\to c_0(v_k)$, for $k \in \N $,  is continuous.   The result then  follows  from the general theory
of (LB)-spaces as  $E_\al = \ind_k c_0 (v_k)$.
\end{proof}

\begin{lemma}\label{L.s1}
Let $\al_n \uparrow \infty $. The following conditions are equivalent.
\begin{itemize}
\item[\rm (i)] $\sup_{n\in\N}\frac{\log n}{\alpha_n}<\infty$.
\item[\rm (ii)] For each $ \g > 0$ there exists $M( \g)\in\N$ such that $\sup_{n\in\N} n^\g e^{-M(\g)\alpha_n}<\infty$.
\item[\rm (iii)] For some $\g > 0$ and $M (\g) \in  \N $ we have $\sup_{n \in \N } n^\g e^{-M (\g) \al_n} < \infty  $.
\end{itemize}
\end{lemma}

\begin{proof}% In this proof we take $v_k(n)=e^{-k\alpha_n}$ for $k,\ n\in\N$.
(i)$\Rightarrow$(ii). Fix any $\g > 0$. By assumption there exists $D>0$ such that $\log n\leq D\alpha_n$ for all $n\in\N$.
Let $M(\g)\in\N$ satisfy $M(\g)\geq \g D$. Then $ \g \log n\leq  \g D\alpha_n\leq M(\g)\alpha_n$ for all $n\in\N$ and hence, $n^\g\leq e^{M(\g)\alpha_n}$ for all $n\in\N$.

(ii)$\Rightarrow$(iii) is clear.

(iii)$\Rightarrow$(i). By assumption  $\sup_{n\in\N} n^\g e^{-M(\g)\alpha_n}<\infty$. So, there exists $D>1$ such that
$n^\g \leq D e^{M(\g)\alpha_n}$ for all $n\in\N$. It follows for each $n\in\N$ that $\frac{\log n}{\alpha_n}\leq \frac{\log D}{\g \alpha_n}+\frac M \g $.
Since $\alpha_n\to \infty$,   we can conclude  that $\sup_{n\in\N}\frac{\log n}{\alpha_n}<\infty$.
\end{proof}

We now turn our attention to the spectrum of $\sC \in \cL (E_\al)$, for which we introduce  the notation
$\Sigma := \{\frac 1 n : n \in \N  \}$ and $\Sigma_0 := \{0\} \cup \Si $. The Cesàro matrix  $\sC$, when acting in $\C^\N $, is similar
to the diagonal matrix $\mbox{diag} ((\frac 1 n ))$. Indeed, $\sC = \Delta \mathrm{diag} ((\frac 1 n )) \Delta $  with $\Delta = \Delta^{-1}
= (\Delta_{nk})_{n, k \in \N } \in \cL (\C^\N )$ the lower triangular matrix where, for each $n \in \N $, $\Delta_{nk} = (-1)^{k-1}\binom
{n-1} {k -1} $, for $1 \le k < n $ and $\Delta_{nk} =0$ if $k > n $, \cite[pp.\,247-249]{H}. Thus  $\sigma_{pt}(\sC ; \C^\N )=\Sigma$ and
each eigenvalue $\frac 1 n $ has multiplicity $1$ with eigenvector $\Delta e_n$, where $e_n := (\delta_{nk})_{k \in \N }$, for $n \in \N $, are the
canonical basis vectors in $\C^\N $. Moreover, $\lambda I -\sC$ is invertible for each $\lambda \in \C \setminus \Si $. If $X $ is a lcHs continuously
contained in $\C^\N $ and $\sC (X) \su X $, then
\begin{equation}\label{22}
    \sigma_{pt} (\sC ; X) = \{\textstyle \frac 1 n :  n \in \N , \; \Delta e_n \in X \} \su \Si .
\end{equation}
In case the space $\vp$ (of all finitely supported vectors in $\C^\N $) is densely contained in $X$, then $\vp \su X'$ and $\Si \su \si_{pt}
(\sC'  ; X' ) \su  \si (\sC ;  X )$, where $\sC'$ is the dual operator of $\sC$. Observe that always $\Delta e_1 = \boldsymbol{1} := (1)_{n \in \N} \in
c_0 (v_1) \su E_\al$ whenever $\al_n \uparrow \infty $. Since $\vp$ is dense in $E_\al$ for \textit{every} $\al$ with $\al_n \uparrow \infty $,
we conclude that \textit{always}
\begin{equation}\label{23}
    1 \in \si_{pt} (\sC ; E_\al) \su \Si \su \si (\sC ; E_\al).
\end{equation}
We point out that $\sC$ does \textit{not} act in the vector space $\vp : = \ind_k \C^k \su \C^\N $ because $e_1 \in \vp $ but $ \sC e_1 = (\frac 1n) \notin \vp$.

\begin{prop}\label{P.spP}
For $\al $ with  $\al_n \uparrow \infty $ the following assertions are equivalent.
\begin{enumerate}
 \item[\rm (i)]
 $ E_\al $  is nuclear.
 \item[\rm (ii)]
  $\sup_{n\in\N}\frac{\log n}{\alpha_n}<\infty$.
  \item[\rm (iii)]
 $\si_{pt}  (\sC ; E_\al) = \Si $ .
 \item[\rm (iv)]
 $\si_{pt} (\sC; E_\al) \setminus \{1\} \ne \emptyset $.
\end{enumerate}
\end{prop}

\begin{proof}
(i) $\Leftrightarrow$ (ii). See the introduction to this section.

(ii) $\Rightarrow$ (iii).
Observe that $\Delta e_m$, for fixed $m \in \N $, behaves asymptotically like $(n^{m-1})_{n \in \N }$, i.e.,
$|(\Delta e_m)| \simeq n^{m-1}$ for $n \to \infty $. By Lemma \ref{L.s1} each $\Delta e_m \in E_\al$ for $m \in \N$.
Hence, \eqref{22} yields that $\si_{pt} (\sC ; E_\al) = \Si$.

(iii) $\Rightarrow$ (iv).
Obvious.

(iv) $\Rightarrow$ (ii).
For this proof select $v_k (n) := e^{-k \al_n}$, $n \in \N $, for each $k \in \N $. By \eqref{23} and the assumption (iv)
there exists $m \in \N $ with $m > 1 $ such that $\frac 1 m \in \si_{pt} (\sC ; E_\al)$, i.e., $\Delta e_m \in E_\al$. As
seen in the proof of (ii) $\Rightarrow$ (iii) we then have $(n^{m-1})_{n \in \N} \in E_\al$. Hence, for some $k \in \N $, $ (n^{m-1})_{n \in \N} \in
c_0 (v_k)$ and so there exists $M > 1$ such that $n^{m-1} v_k (n) = n^{m-1} e^{-k \al_n} \le M $ for all $n \in \N $. It follows
from Lemma \ref{L.s1} that (ii) holds.
\end{proof}

\begin{prop}\label{P.sp0}
Let $\al_n \uparrow \infty $. The following conditions are equivalent.
\begin{itemize}
\item[\rm (i)]  $\sup_{n\in\N}\frac{\log n}{\alpha_n}<\infty$, i.e., $E_\al $ is nuclear.
\item[\rm (ii)]  $ \sC \in \cL (E_\al)$ is invertible, i.e., $0\in \rho(\sC ;  E_\alpha)$.
\end{itemize}
\end{prop}

\begin{proof} Note that $\sC\colon \C^\N\to\C^\N$ is bijective with inverse $\sC^{-1}\colon \C^\N\to\C^\N$  given by
\begin{equation}\label{24}
    \sC^{-1}y=(ny_n-(n-1)y_{n-1}) , \qquad  y=(y_n)\in\C^\N ,
\end{equation}
with $y_0 := 0$.   Accordingly, $0\not\in \sigma(\sC;  E_\alpha)$ if and only if $\sC^{-1}\colon E_\alpha\to E_\alpha$ is
continuous if and only if for each $k\in\N$ there exists $l \ge k$ such that $\sC^{-1}\colon c_0(v_k)\to c_0(v_l)$ is continuous.

For the rest of the  proof we select  $v_k(n) :=e^{-k\alpha_n}$ for $k,\ n\in\N$, i.e., $s_k := e^k$.

(i)$\Rightarrow$(ii). By Lemma \ref{L.s1}  there exists $m\in\N$ with $D:=\sup_{n\in\N}ne^{-m\alpha_n}<\infty$.
Fix $k\in\N$ and set $l:=m+k$. Let  $y=(y_n)\in c_0(v_k)$. For each   $n\in\N$, we have
\begin{eqnarray*}
&&v_l (n) (\sC^{-1} y ) = e^{-l\alpha_n} |ny_n-(n-1)y_{n-1}| \leq  e^{-l\alpha_n} n|y_n|+e^{-l\alpha_{n-1}} (n-1)|y_{n-1}|\\
&& \qquad \qquad \qquad \leq  D(e^{-k\alpha_n} |y_n|+e^{-k\alpha_{n-1}} |y_{n-1}|)\leq 2Dq_k(y).
\end{eqnarray*}
 Forming the supremum relative to $n \in \N $ yields  $q_l(\sC^{-1}y)\leq 2Dq_k(y)$ for all $y\in c_0(v_k)$.
 Accordingly, $\sC^{-1}\colon c_0(v_k)\to c_0(v_l)$ is continuous. Since $k\in\N$ is arbitrary, it follows that $\sC^{-1}\colon E_\alpha\to E_\alpha$
 is continuous and so $0 \in \rho (\sC ; E_\alpha)$.

(ii)$\Rightarrow$(i). By assumption  $\sC^{-1}\colon E_\alpha\to E_\alpha$ is continuous. So, there exists $l\in\N$ such that $\sC^{-1}\colon c_0(v_1)\to c_0(v_l)$
is continuous, that is, there exists $D>1$ such that  $q_l(\sC^{-1}y)\leq Dq_1(y)$ for all $y\in c_0(v_1)$. Since $\sC^{-1}e_n= n e_n - n e_{n+1}$ and $q_l(\sC^{-1}e_n)=\max\{nv_l(n),nv_l(n+1)\}=nv_l(n)=ne^{-l\alpha_n}$, with $q_1(e_n)=v_1(n)=e^{-\alpha_n}$,  for all $n\in\N$, it follows  that
$ne^{-l\alpha_n}\leq D e^{-\alpha_n}$, for $n \in \N $. Hence,  $ne^{(1-l)\alpha_n}\leq D$, for $n \in \N $, which   implies that $\sup_{n\in\N}\frac{\log n}{\alpha_n}<\infty$.
\end{proof}

The operator of differentiation $D$ acts on $\C^\N$ via
$$
D (x_1, x_2, x_3 , \ldots ) := (x_2, 2 x_3, 3x_4, \ldots ) , \qquad x = (x_n) \in \C^\N .
$$
Clearly $D \in \cL (\C^\N )$. According to \eqref{24} and a routine calculation the inverse operator $\sC^{-1} \in \cL (\C^\N )$ is given by
\begin{equation}\label{25}
    \sC^{-1} = (I - S_r) D S_r ,
\end{equation}
where $S_r \in \cL (\C^\N )$ is the right-shift operator, i.e., $S_r x := (0, x_1, x_2, \ldots )$ for $x \in \C^\N$. Fix $k \in \N $. Since
$v_k $ is decreasing on $\N $, it follows that
$$
q_k (S_r x ) := \sup_{n  \in \N} v_k (n+1) |x_n|  \le \sup_{n\in \N} v_k (n) |x_n| = q_k (x) , \qquad x \in c_0 (v_k) .
$$
Hence, $S_r : c_0 (v_k) \to c_0 (v_k)$ is continuous for each $k \in \N $ which implies (for \textit{every} $\al_n \uparrow \infty $) that
$S_r \in \cL (E_\al)$. Moreover, Proposition \ref{P.sp0} shows that $\sC^{-1} \in \cL (E_\al)$ if and only if $E_\al$ is nuclear.
The identity \eqref{25} suggests there should be  a connection between the nuclearity of $E_\al$ and the continuity of $D$
on $E_\al$. The following result addresses this point. Recall that $E_\al$ is \textit{shift stable} if $\limsup_{n \to \infty }
\frac{\al_{n+1}}{\al_n} < \infty$, \cite{V}.

\begin{prop} \label{pr25}
For $\al $ with $\al_n \uparrow \infty $ the following assertions are equivalent.
\begin{enumerate}
  \item[\rm (i)]
$D (E_\al) \su E_\al$, i.e., $D$ acts in $E_\al$.
  \item [\rm (ii)]
The differentiation operator $D \in \cL (E_\al)$.
  \item [\rm (iii)]
For every $k \in \N $   there exists $l \in \N $ with $l > k $ such that $D: c_0 (v_k) \to c_0 (v_l)$ is continuous.
\item[\rm (iv)]
For every $k \in \N $ there exist $l \in \N $ with $l > k $ and $M>0$ such that
$$
n v_l (n) \le M v_k (n+1) , \qquad n \in \N .
$$
\item[\rm (v)]
The space $E_\al$ is both nuclear and shift stable.
\end{enumerate}
\end{prop}

\begin{proof}
(i)$\Leftrightarrow$(ii) is immediate from the closed graph theorem for (LB)-spaces, \cite[Theorem 24.31 and Remark 24.36]{MV}.

(ii)$\Leftrightarrow$(iii) is a general fact about continuous linear operators between  (LB)-spaces.

(iii)$\Rightarrow$(iv). Fix $k \in \N$. By (iii) there exists $l \in \N $ with $l > k $ such that $D : c_0 (v_k) \to c_0 (v_l)$ is continuous.
Hence, there is $M >0 $ satisfying
$$
q_l (D x ) = \sup_{n \in \N } v_l (n) |(D x)| \le M q_k (x) = M \sup_{n \in \N } v_k (n) |x_n| ,  \qquad x \in c_0 (v_k) .
$$
For each $j \in \N $ with $j \ge 2 $ substitute $x := e_j$ in the previous inequality (noting  that $D x = D e_j = (j-1)e_{j-1} $)   yields
$(j-1) v_l (j-1) \le M v_k (j)$. Since $j \ge 2 $ is arbitrary, this is precisely (iv).

(iv)$\Rightarrow$(iii). Given any $k \in \N $ select $l > k $ and $M > 0 $ which satisfy (iv). Fix $x \in c_0 (v_k)$. Then, for each
$n \in \N $, we have via (iv) that
$$
v_l (n) |(D x )| = n v_l (n) |x_{n+1}| \le M v_k (n+1)  .
$$
Forming the supremum relative to $n \in \N $ of both sides of this inequality yields
$$
q_l (D x ) \le M q_k (x) , \qquad x \in c_0 (v_k),
$$
which is precisely (iii).

(iv)$\Rightarrow$(v). For $k = 1 $, condition (iv) ensures the existence of $l > 1 $ and $M> 1 $ such that
\begin{equation}\label{26}
    n v_l (n ) \le M v_1 (n+1) \le M v_1 (n) , \qquad n \in \N .
\end{equation}
For the remainder of the proof of this proposition, choose $s_k : = e^k$ for $k \in \N $. It follows from \eqref{26}
that $n e ^{-l \al_n} \le M e ^{-\al_n}$  for all $n \in \N $. By Lemma \ref{L.s1}  one can conclude that $E_\al$ is \textit{nuclear}.

To prove that $E_\al$ is shift stable observe that the left-inequality in \eqref{26} is $n e ^{-l \al_n} \le Me ^{-\al_{n+1}}$ for $n \in \N $. Taking
logarithms and rearranging yields
$$
\frac{\al_{n+1}}{\al_n} \le l + \frac{\log (M)}{\al_n} - \frac{\log (n)}{\al_n} ,  \qquad n \in \N .
$$
Since $\sup_{n \in \N } \frac{\log (n)}{\al_n} < \infty $ (as $E_\al$ is nuclear) and $\sup_{n \in \N } \frac{\log (M)}{\al_n} < \infty $   it follows
that $\sup_{n \in \N} \frac{\al_{n+1}}{\al_n} < \infty $, i.e., $E_\al$ is \textit{shift-stable}.

(v)$\Rightarrow$(iv). Fix $k \in \N $. Since $E_\al$ is shift stable, there exists $h \in \N $ such that $\al_{n+1} \le h \al_n$ for $n \in \N $. Because of
the nuclearity of $E_\al$, Lemma \ref{L.s1}  implies the existence of $M \in \N $ which satisfies $L: = \sup_{n \in \N } n e^{- M \al_n} < \infty$.
Set $l := M + h k$. Then $l \in \N $ and, for each $n \in \N $, it follows that
$$
n v_l (n) = n e^{- l  \al_n} = n e^{-M  \al_n} e^{-h k  \al_n} \le L e ^{-k (h \al_n) } \le Le^{-k  \al_{n+1}} = L v_k (n+1)  .
$$
This is precisely condition (iv).
\end{proof}

\begin{remark}\label{r26} \rm
(i)  There exist nuclear spaces $E_\al $ for which $D$ is \textit{not} continuous on $E_\al$. Let $ \al_n := n^n$ for $n \in \N $.
Then $E_\al$ is nuclear but, not shift stable. Proposition \ref{pr25} implies that $D \notin \cL (E_\al)$. On the other hand, for
$ \al_n := \log (\log (n))$ for $n \ge 3 $, the space $E_\al$ is shift stable but, not nuclear; again $D \notin \cL (E_\al)$.

(ii) Because $v_1 \downarrow 0$, it is clear that $\ell_\infty \su \ell_\infty (v_1) \su E_\al := \ind_k \ell_\infty (v_k)$ for every $\al $ with $\al_n \uparrow \infty $.
Accordingly, if $x_\lambda := (\frac{\lambda^{n-1}}{(n-1)!})_{n \in \N }$ for $\lambda \in \C$, then clearly $\{x_\lambda : \lambda \in \C \}
\su \ell_\infty $ and so $\{x_\lambda : \lambda \in \C\} \su E_\al$. Since $D x_\lambda = \lambda x_\lambda $ for each $\lambda \in \C$,
we have established (via Proposition \ref{pr25}) the following fact.

\textit{Let $\al$ with $\al_n \uparrow \infty $ be a sequence such that $E_\al$ is both nuclear and shift stable. Then $D \in \cL (E_\al)$ and}
$$
\si_{pt} (D; E_\al ) = \si (D; E_\al) = \si^* (D; E_\al ) = \C .
$$

In order to determine $\si (\sC ; E_\al)$ we require some further preliminaries. Define the continuous function $a : \C \setminus \{0\} \to \R$
by $a (z) := \mbox{Re} (\frac 1 z)$ for $z \in \C \setminus \{0\}$. The following result is a refinement of \cite[Lemma 7]{R}.
\end{remark}

\begin{lemma}\label{L27}
Let  $\lambda\in\C \setminus \Si_0$. Then there exists $\delta = \delta_\lambda> 0 $ and positive constants
$d_\delta, D_\delta $ such that $\ov{B (\lambda, \delta )} \cap \Si_0 = \emptyset$ and
\begin{equation}\label{27}
\frac{d_\delta }{N^{a(\mu)}} \le \prod^N_{n=1} \big |1- \textstyle{ \frac{1}{n \mu} } \big |\leq \frac{D_\delta }{N^{a(\mu)}} , \qquad \forall N \in \N , \  \mu \in B (\lambda, \delta).
\end{equation}
\end{lemma}

\begin{proof}
Fix $\lambda \in \C \setminus \Si_0$ and write $\frac 1 \lambda = \al + i \beta $ with $\al, \beta \in \R $, i.e., $\al =  a (\lambda)$. Observe that
$$
1- \frac{2 \al}{n} + \frac{(\al^2 + \beta^2)}{n^2} = \big(1- \frac \al n \big)^2 + \frac{\beta^2}{n^2} > 0 , \qquad n \in \N .
$$
Using the inequality $(1+x)  \le e^x $ for $x \in \R $ we conclude that $(1+x)^{1/2} \le e^{x/2}$ for all $x \ge - 1 $. In particular, for
$x := -\frac{ 2 \al}{n} + \frac{(\al^2 + \beta^2)}{n^2}$ it follows that
$$
\Big ( 1 -\frac{ 2 \al}{n} + \frac{(\al^2 + \beta^2)}{n^2}\Big )^{1/2} \le \exp \Big(  -\frac{ \al}{n} + \frac{(\al^2 + \beta^2)}{2n^2} \Big ) , \qquad
n \in \N .
$$
Fix $N \in \N $. Since $\sum^N_{n=1} \frac{1}{n^2} < 2 $, we conclude that
\begin{eqnarray*}
   && \prod^N_{n=1} \Big|1- \frac{1}{n \lambda } \Big|  = \prod^N_{n=1} \Big( 1- \frac{ 2 \al}{n} + \frac{(\al^2 + \beta^2)}{n^2}\Big )^{1/2}\\
   && \le \exp \Big ( \sum^N_{n=1} - \frac{\al}{n} + \frac{(\al^2 + \beta^2)}{2 n^2}  \Big) \le \exp (\al^2 + \beta^2) \exp \Big(- \al \sum^N_{n=1} \frac 1 n \Big) \\
   && = \exp \big( \frac{1}{|\lambda |^2}\big) \exp \Big(- \al \sum^N_{n=1} \frac 1n \Big).
\end{eqnarray*}
By considering separately the cases when $\al \le 0 $ and $\al > 0$ and employing the inequalities
\begin{equation}\label{28}
    \log (k+1) \le \sum^k_{n=1} \frac 1 n \le 1 + \log (k ), \qquad k \in \N ,
\end{equation}
it turns out that
$$
\exp \Big(- \al \sum^N_{n=1} \frac 1 n  \Big)  \le \frac{e^{|a (\lambda)|}}{N^{a (\lambda)}} \le \frac{e^{1/ |\lambda |}}{N^{a (\lambda)}}.
$$
Accordingly, we have that
\begin{equation}\label{29}
    \prod^N_{n=1} \big|1- \frac{1}{n \lambda } \big| \le \frac{\exp(\frac{1}{|\lambda|} + \frac{1}{|\lambda |^2})}{N^{a (\lambda)}} , \qquad N \in \N .
\end{equation}

From above, for each $n  \in \N $, we have $|1- \frac{1}{n \lambda}|^{-1} = (1+x_n)^{-1/2}$, where $x_n := - \frac{2 \al}{n} + \frac{(\al^2 + \beta^2)}{n^2}$
satisfies $x_n > - 1$. Applying Taylor's  formula to the function $f (x) = (1+x)^{- 1/2}$ for $ x> -1$ yields, for each $n \in \N $, that
\begin{eqnarray*}
  (1+x_n)^{-1/2} &=&  f (0) + f' (0)x_n + \frac{f'' (\theta_n x_n)}{2!} x^2_n \\
   &=& 1-   \textstyle \frac 1 2 x_n + \frac 3 4 (1+ \theta_n x_n)^{-5/2} x^2_n
\end{eqnarray*}
for some $\theta_n \in (0,1)$. Substituting for  $x_n $ its  definition  and rearranging we get
$$
(1+x_n)^{-1/2} = 1 + \frac{ \al}{n} - \frac{(\al^2 + \beta^2)}{2 n^2} +  {\textstyle \frac 3 4} (1- \theta_n + \theta_n |1-  {\textstyle \frac{1}{\lambda n}}|)^{-5/2}
\Big(-\frac{ 2 \al}{n} + \frac{(\al^2 + \beta^2)}{n^2}\Big)^2 ,
$$
for each $n \in \N $. Defining $d (\lambda) := \mbox{dist} (\lambda, \Si_0)  \le |\lambda|$ we have
$$
\big| 1-  \textstyle{ \frac{1}{\lambda n}}   \big|  = \frac{1}{|\lambda |} \cdot \big |\lambda-  \textstyle{\frac 1 n } \big| \ge \frac{d (\lambda )}{|\lambda |}, \qquad
n \in \N .
$$
Hence, for each $n \in \N $, it follows that
$$
1- \theta_n + \theta_n \big| 1-  \textstyle{\frac{1}{\lambda n}}   \big | \ge 1- \theta_n + \theta_n \frac{d (\lambda)}{|\lambda |} \ge \min
\Big\{ 1, \frac{d (\lambda )}{|\lambda |}\Big \} = \frac{d (\lambda )}{|\lambda |} ,
$$
where we have used the inequality
$$
1-x + \gamma x \ge \min \{1, \gamma\}, \qquad \forall \gamma \in \R , \ x \in [0,1] .
$$
Accordingly,
$ \big (1- \theta_n + \theta_n \big| 1- \frac{1}{\lambda n}   \big | \big )^{-5/2} \le \big (\frac{|\lambda |}{d (\lambda )} \big )^{5/2} $,
for $ n \in \N $,
which implies (see above), for each $n \in \N $, that
\begin{eqnarray*}
\big|1- {\textstyle \frac{1 }{n \lambda } }\big |^{-1}   &\le & 1 + \frac \al n + \frac{1}{n^2} \Big ( - \frac{(\al^2 + \beta^2)}{2} + \frac 3 4 \Big ( \frac{|\lambda|}{d (\lambda)}\Big )^{5/2}
\Big ( -2 \al + \frac{(\al^2 + \beta^2)}{n}\Big )^2\Big ) \\
   &\le &   1 + \frac \al n + \frac{3}{4n^2}  \Big ( \frac{|\lambda|}{d (\lambda)}\Big )^{5/2} \big ( 2 |\al| + \al^2 + \beta^2 \big )^2 .
\end{eqnarray*}
But, $ (2 |\al| + \al^2 + \beta^2)^2 \le \big (\frac{2}{|\lambda|} + \frac{1}{|\lambda |^2}\big )^2 \le 4  \big (\frac{1}{|\lambda|} + \frac{1}{|\lambda |^2}\big )^2 $
and so
$$
\big|1- \frac{1 }{n \lambda }\big |^{-1}  \le 1 + \frac \al n + \frac{D (\lambda )}{n^2} , \qquad n \in \N ,
$$
with $D (\lambda) := \frac{3 (1+|\lambda |)^2}{|\lambda |^{3/2} (d (\lambda ))^{5/2}}$. Accordingly, for fixed $N \in \N$, we have
\begin{eqnarray*}
\prod^N_{n=1} \Big |1-  {\textstyle \frac{1}{\lambda n}}\Big |^{-1} & \le & \prod^N_{n=1} \Big(1+ \frac \al n + \frac{D (\lambda )}{n^2} \Big)
\le \exp \Big ( \al \sum^N_{n=1} \frac 1 n \Big ) \exp \Big ( D (\lambda) \sum^N_{n=1} \frac {1}{n^2}\Big) \\
&\le & e^{2 D (\lambda )}   \exp \Big ( \al \sum^N_{n=1} \frac {1}{n}\Big) .
\end{eqnarray*}
By considering separately the cases when $\al < 0$ and $\al \ge 0$ and applying \eqref{28} yields
$$
 \exp \Big ( \al  \sum^N_{n=1} \frac {1}{n}\Big) \le e^{|\al|} N^{\al} \le e^{\frac{1}{|\lambda |}} N^{a (\lambda)}  .
$$
Accordingly, $\prod^N_{n=1}  |1- \frac{1}{\lambda n}|^{-1}  \le N^{a (\lambda )} \exp (2 D (\lambda) + \frac{1}{|\lambda |})$ and hence,
\begin{equation}\label{210}
\frac{\exp (- \frac{1 }{|\lambda |}- 2 D (\lambda ))}{N^{a (\lambda)}} \le \prod^N_{n=1} \big |1-  \textstyle{\frac{1}{n \lambda }}\big | , \qquad N \in \N .
\end{equation}
It follows from \eqref{29} and \eqref{210}, for any given $\lambda \in \C \setminus \Si_0$, that
\begin{equation}\label{211}
    \frac{u (\lambda )}{N^{a(\lambda )}} \le \prod^N_{n=1} \big |1-   {\textstyle\frac{1}{\lambda n}}\big |\le \frac{v (\lambda)}{N^{a (\lambda)}}, \qquad N \in \N ,
\end{equation}
where $v (\lambda ) := \exp (\frac{1}{|\lambda|} + \frac{1}{|\lambda|^2})$ and $u (\lambda) := \exp (- \frac{1}{|\lambda |} - \frac{6 (1+|\lambda|^2)}
{|\lambda|^{3/2} (d (\lambda))^{5/2}})$.

Fix now a point $\lambda \in \C \setminus \Si_0$ and choose any $\delta > 0$ satisfying $\ov{B (\lambda, \delta )} \cap \Si_0 = \emptyset$.
According to \eqref{211} we have
\begin{equation}\label{212}
  \frac{u (\mu )}{N^{a(\mu )}} \le \prod^N_{n=1} \big |1-  {\textstyle\frac{1}{ n\mu }}\big |\le \frac{v (\mu)}{N^{a (\mu)}}, \qquad \forall N \in \N    , \
  \mu \in \ov{B (\lambda, \delta )} .
\end{equation}
By the continuity (and form) of the functions $u$ and $v$ on $\C \setminus \Si_0$ and the compactness of the set $\ov{B (\lambda, \delta )} \su
(\C \setminus \Si_0)$ it follows that $D_\delta := \sup \{v (\mu) : \mu \in \ov{B (\lambda, \delta )}\} < \infty $ and $d_\delta := \inf
\{u (\mu) : \mu \in \ov{B (\lambda, \delta )}\} > 0 $. It is then clear that \eqref{24} follows from \eqref{212}.
\end{proof}

\begin{lemma} \label{L28}
Let $w = (w_n)$  be any strictly positive, decreasing sequence. Then
\begin{equation}\label{213}
\si (\sC ; c_0 (w)) \su \{ \lambda\in \C : |\lambda -  \textstyle \frac 1 2 | \le \frac 1 2 \}  .
\end{equation}
Moreover, for each $\lambda \in \C $ satisfying $|\lambda - \frac 1 2 |> \frac 1 2 $ there exist constants $\delta_\lambda > 0$ and
$M_\lambda > 0$ such that
$$
\|(\mu I-\sC)^{-1}\|_{op} \le \frac{M_\lambda }{1- a (\mu)} , \qquad \mu \in B (\lambda , \delta_\lambda) ,
$$
where $\|\boldsymbol{\cdot}\|_{op}$  denotes the operator norm in $\cL (c_0 (w))$.
\end{lemma}

\begin{proof}
According to \cite[Corollary 2.3(i)]{ABR-9} the Cesàro operator $\sC : c_0 (w) \to c_0 (w)$ is continuous. Then Corollary 3.6 of
\cite{ABR-9} implies that \eqref{213} is satisfied.

Set $A:= \{\lambda \in \C : |\lambda - \frac 1 2 | \le \frac 1 2 \}$ and fix $\lambda \in \C \setminus A $. Define $\delta_\lambda :=
\frac 1 2 \mbox{dist} (\lambda , A) > 0 $ and $C_\lambda := \ov{B (\lambda, \delta )} $, in which case \eqref{213} implies that $\mbox{dist}
(C_\lambda , \si (\sC ; c_0 (w))) \ge \mbox{dist} (C_\lambda, A) = \delta_\lambda$. According to Lemma 6.11 of \cite[p.\,590]{DSI}
there is a constant $K > 0 $ such that (setting $\ve := \delta_\lambda $ in that lemma)
\begin{equation}\label{214}
    \| (\mu I - \sC)^{-1}\|_{op} < \frac{K}{\delta_\lambda }, \qquad \mu \in C_\lambda  .
\end{equation}
Now, each $\mu \in B (\lambda, \delta_\lambda )$ satisfies $a (\mu) < 1 $, \cite[Remark 3.5]{ABR-9}, and so
\begin{equation}\label{215}
    \frac{K}{\delta_\lambda } = \frac{K \delta_\lambda^{-1}(1-a (\mu))}{1- a (\mu)} \le \frac{K \delta_\lambda^{-1} (1+ \frac{1}{|\mu|})}
{1- a(\mu)} \le \frac{M_\lambda }{1-a (\mu)} ,
\end{equation}
where $M_\lambda := \sup \{ \frac{K}{\delta_\lambda} (1+ \frac{1}{|z|}) : z \in C_\lambda \} < \infty $ as the set $C_\lambda \su (\C
\setminus \{0\})$ is compact and the function $z \mapsto \frac{K}{\delta_\lambda } (1 + \frac{1}{|z|})$ is continuous on $\C \setminus \{0\}$.
The desired inequality follows from \eqref{214} and \eqref{215}.
\end{proof}

Recall that a Hausdorff inductive limit $E=\ind_k E_k$ of Banach spaces is called \textit{regular} if every
 $B \in \cB (E)$   is contained and bounded in some step $E_k$. In particular, for every $\al$ with $\al_n \uparrow \infty$
 the space $E_\al = \ind_k c_0 (v_k)$ is regular, \cite[Proposition 25.19]{MV}.

\begin{prop}\label{P.spectrum}
Let $\al$ satisfy   $\al_n \uparrow \infty $ with $E_\al$ nuclear. Then
\begin{itemize}
\item[\rm (i)] $\sigma(\sC;  E_\alpha)=\sigma_{pt}(\sC ; E_\alpha)= \Sigma $, and
\item[\rm (ii)] $\sigma^*(\sC ; E_\alpha)=\sigma(\sC ; E_\alpha)\cup\{0\}= \Sigma_0$.
\end{itemize}
\end{prop}

\begin{proof} By Proposition \ref{P.spP} we have $ \Si =  \sigma_{pt}(\sC ; E_\alpha)  \su \sigma(\sC  ;E_\alpha)$ and hence,
$$
\Si_0 = \ov{\Si} \su \ov{\si (\sC ; E_\al)} \su \si^* (\sC ; E_\al) .
$$
Moreover,  Proposition \ref{P.sp0} yields  $0\not\in\sigma(\sC ; E_\alpha)$.   So, it remains to show that
$ (\C \setminus \Si_0) \su \rho^* (\sC ; E_\al)$.
To this end, we  need to show,    for each $\lambda\in \C\setminus  \Si_0 $, that  there exists $\delta>0$ with the property that
 $(\sC-\mu I)^{-1}\colon E_\alpha\to E_\alpha$ is continuous for each $\mu \in B (\lambda, \delta )$ and the set $\{(\sC-\mu I)^{-1}\colon \mu \in B (\lambda, \delta)\}$
 is equicontinuous in $ \cL (E_\alpha)$.  We recall  that  $(\sC-\mu I)^{-1}\colon \C^\N\to \C^\N$ exists in $\cL(\C^\N)$ for each
 $\mu \in \C \setminus \Si $.

For  this proof we select the weights  $v_k(n)=e^{-k\alpha_n}$,   $ n\in\N$, for each $k \in \N $.
Fix $\lambda\in \C\setminus   \Si_0 $.  First, choose $\delta_1>0$ such that  $\ov{B (\lambda, \delta_1)} \cap \Si_0 = \emptyset $.
 Later  $\delta>0$ will be selected   in such a way that $0<\delta<\delta_1$.

According to  Lemma \ref{Lereg} in the Appendix it suffices  to find  a $\delta>0$ satisfying  the following condition:
for each $k \in \N $ there exists $l \in \N $ with $l \ge k $ and $D_k > 0 $ such that
\begin{eqnarray}\label{eq.condequi}
 q_l((\sC-\mu I)^{-1}x)\leq D_kq_k(x), \qquad \forall \mu \in B (\lambda, \delta ), \ x \in c_0 (v_k).
\end{eqnarray}

\textit{Case} (i).
Suppose that   $\left|\lambda-\frac{1}{2}\right|>\frac{1}{2}$ (equivalently,  $ a (\lambda)<1$,  \cite[Remark 3.5]{ABR-9}).
To establish the   condition \eqref{eq.condequi}  we proceed as  follows. Fix $k \in \N $.
Since $a (\lambda)<1$, we can select $\ve>0$ such that $ a(\lambda)<1-\ve$. By continuity of  the function
 $a \colon \C\setminus\{0\}\to \R$    there exists $\delta_2 \in (0,\delta_1)$
 such that  $a (\mu)<1-\ve$  for all  $\mu\in \ov{B(\lambda, \delta_2)}$. Applying Lemma \ref{L28} (with $v_k$ in place of $w$),
 it follows   that there exist $\delta \in (0, \delta_2)$ and $M_{k, \lambda} >0$ satisfying
\[
q_k((\sC-\mu I)^{-1}x)\leq \frac{M_{k, \lambda}}{1- a(\mu)}q_k(x)\leq \frac{M_{k, \lambda}}{\ve}q_k(x)
\]
for all $\mu\in \ov{B(\lambda, \delta)}$ and $x\in c_0(v_k)$. So, inequality  \eqref{eq.condequi} is then satisfied with $l: =k$
and $D_k: =\frac{M_{k, \lambda }}{\ve}$. Since $k \in \N $ is arbitrary, condition \eqref{eq.condequi} holds.

\textit{Case} (ii).
Suppose  now  that $\left|\lambda-\frac{1}{2}\right|\leq\frac{1}{2}$ (equivalently, $a(\lambda)\geq 1$, \cite[Remark 3.5]{ABR-9}).
We recall the formula for the inverse operator $(\sC-\mu I)^{-1}\colon \C^\N\to \C^\N$ whenever $\mu\not \in \Si_0$, \cite[p.\,266]{R}.
For $n\in\N$ the $n$-th row of the matrix for $(\sC-\mu I)^{-1}$ has the entries
\[
\frac{-1}{n\mu^2\prod_{k=m}^n\left(1-\frac{1}{\mu k}\right)},\quad 1\leq m<n,
\]
\[
\frac{n}{1-n\mu}=\frac{1}{\frac{1}{n}-\mu}, \quad m=n,
\]
and all the other entries in row $n$ are equal to $0$. So,
we can write
\begin{equation}\label{e:dec1}
(\sC-\mu I)^{-1}=D_\mu-  \textstyle \frac{1}{\mu^2}E_\mu, \qquad \mu \in \C \setminus \Si_0 ,
\end{equation}
where the diagonal operator $D_\mu=(d_{nm}(\mu))_{n,m\in\N}$ is given by $d_{nn}(\mu):=\frac{1}{\frac{1}{n}-\mu}$ and $d_{nm}(\mu):=0$ if $n\not=m$. The operator $E_\mu=(e_{nm}(\mu))_{n,m\in\N}$ is then the lower triangular matrix with $e_{1m}(\mu)=0$ for all $m\in\N$, and for every $n\geq 2$ with $e_{nm}(\mu):=\frac{1}{n\prod_{k=m}^n\left(1-\frac{1}{\mu k}\right)}$ if $1\leq m<n$ and $e_{nm}(\mu):=0$ if $m\geq n$.

Since $d_0 (\lambda ):={\rm dist}(\ov{B(\lambda, \delta_1)}, \Si_0 )>0$, we have   $|d_{nn} (\mu )|\leq \frac{1}{d_0 (\lambda)}$ for all
$\mu\in \ov{B(\lambda, \delta_1)}$ and  $n \in\N$. Fix $k \in \N$. Then, for every  $x\in c_0(v_k)$ and $\mu\in \ov{B(\lambda, \delta_1)}$, we have
\[
q_k(D_\mu(x))=\sup_{n\in\N} |d_{nn}(\mu)x_n|v_k(n)\leq \frac{1}{d_0 (\lambda )}\sup_{n\in\N}|x_n|v_k(n)=\frac{1}{d_0 (\lambda )}q_k(x).
\]
So,  $\{D_\mu : \mu\in \ov{B(\lambda, \delta_1)}\} \su \cL (c_0 (v_k)) $. Moreover,  for every $l\in\N$ with $l\geq k$  it follows that
\begin{equation}\label{eq.primo}
q_l(D_\mu(x))\leq q_k(D_\mu(x))\leq\frac{1}{d_0 (\lambda )}q_k(x), \quad \forall x  \in c_0 (v_k), \quad \mu \in \ov{B (\lambda, \delta_1)}.
\end{equation}

 Via \eqref{e:dec1} it  remains to investigate the operator $E_\mu\colon E_\alpha\to E_\alpha$ in order to  show the validity of condition
 \eqref{eq.condequi} for $(\sC - \mu I )^{-1}$. To this end  we first observe, for each $k\in\N$, that $c_0(v_k)$ is
isometrically isomorphic to $c_0$ via the linear multiplication operator
 $\Phi_k\colon c_0(v_k)\to c_0$ given by $\Phi_k(x):=(v_k(n)x_n)$, for $x=(x_n)\in c_0(v_k)$.
Of course, each $\Phi_k $  is also a bicontinuous isomorphism of $\C^\N $ onto $\C^\N $.
 So,  it suffices to show,    for every  $k\in\N$, that  there exist $l\in\N$ with $l\geq k$ and $D_k>0$ such that
 $\|\Phi_l E_\mu \Phi_{k}^{-1}x\|_0 \leq D_k\|x\|_0 $ for all $x\in c_0$ and $\mu\in\ov{B(\lambda, \delta_1)}$;  here $\| \boldsymbol{\cdot} \|_0$ denotes the usual norm of $c_0$.
For each $k,\ l\in\N$ with $l\geq k$, define  $\tilde{E}_{\mu,k,l}:=\Phi_l E_\mu \Phi_{k}^{-1} \in \cL (\C^\N)$, for $\mu \in \C \setminus \Si_0$.

Fix $k\in\N$. For each $l\geq k$ the operator $\tilde{E}_{\mu ,k,l}$, for $\mu \in B (\lambda, \delta_1)$,
 is  the restriction to $c_0$ of
\[
\tilde{E}_{\mu, k, l } (x) =\big ( (\tilde{E}_{\mu,k,l}(x))\big) = \Big (v_l(n)\sum_{m=1}^{n-1}\frac{e_{nm}(\mu)}{v_{k}(m)}x_m \Big),
\quad  x = (x_n ) \in \C^\N ,
\]
with $(\tilde{E}_{\mu ,k,l}(x))_1:=0$.
Moreover,
observe that $\tilde{E}_{\mu,k,l}=(\tilde{e}^{k,l}_{nm}(\mu))_{n,m\in\N}$ is the lower triangular matrix given by $\tilde{e}^{k,l}_{1m}(\mu)=0$ for $m\in\N$ and $\tilde{e}^{k,l}_{nm}(\mu)=\frac{v_l(n)}{v_{k}(m)}e_{nm}(\mu)$ for $n\geq 2$ and $1\leq m< n$.

So, it suffices  to verify, for some $l \ge k $ and $\delta > 0 $,    that $\tilde{E}_{\mu, k , l } \in \cL (c_0 )$ for $\mu \in B (\lambda, \delta)$ and
$\{\tilde{E}_{\mu,k,l}\colon \mu\in B(\lambda, \delta) \}$ is equicontinuous in $\cL(c_0)$. To prove this  first observe from the definition
of $e_{n m }(\mu )$  that   Lemma \ref{L27} implies, for every $l\geq k$, every  $m,\ n\in\N$ and all $\mu\in \ov{B(\lambda, \delta_2)}$ that
\begin{equation}\label{eq.stimap}
|\tilde{e}^{k,l}_{nm}(\mu)|=\frac{v_l(n)}{v_{k}(m)}|e_{nm} (\mu ) |\leq D'_\lambda \frac{n^{a(\mu)-1}v_l(n)}{m^{a(\mu)}v_{k}(m)} ,
\end{equation}
for some constant $D'_\lambda > 0 $ and $\delta_2 \in (0, \delta_1)$. Because the function
$a \colon \C\setminus\{0\}\to \R$  is continuous,
 there exists $\delta\in (0,\delta_2)$ such that   $a(\lambda) -\frac{1}{2}< a(\mu)< a(\lambda) +\frac{1}{2}$,
 for all $\mu \in \ov{B (\lambda, \delta)}$. This implies, for each  $\mu\in \ov{B(\lambda, \delta)}$ that  $a(\mu)> a(\lambda) -\frac{1}{2}\geq \frac{1}{2}$;
 recall that $ a (\lambda ) \ge 1 $.      Let $ c:=\max\{2, a(\lambda) +\frac{1}{2}\}$.
According to   Lemma \ref{L.s1}  there exists $t\in\N$ such that $S_\lambda :=\sup_{n\in\N}n^c e^{-t\alpha_n}<\infty$.
Set $l:=k+t$.   By \eqref{eq.stimap} and the fact that $\tilde{e}^{k,l}_{nm} (\mu) =0$ for $1 \le m < n $,
 it follows  for every $n\in\N$  and $\mu\in \ov{B(\lambda, \delta)}$ that
\begin{eqnarray*}
& & \sum_{m=1}^\infty |\tilde{e}^{k,l}_{nm}(\mu)|=\sum_{m=1}^{n-1} |\tilde{e}^{k,l}_{nm}(\mu)|\leq D'_\lambda n^{a(\mu)-1}v_l(n)\sum_{m=1}^{n-1}\frac{1}{m^{a(\mu)}v_{k}(m)}\\
& & =D'_\lambda n^{a(\mu)-1}e^{-l\alpha_n}\sum_{m=1}^{n-1}\frac{e^{k\alpha_m}}{m^{a(\mu)}}\leq D'_\lambda n^{a(\mu)-1}e^{-l\alpha_n}\sum_{m=1}^{n-1}e^{k\alpha_m}\\
& & \leq D'_\lambda n^{a(\mu)-1}e^{-l\alpha_n}(n-1)e^{k\alpha_n}\leq   D'_\lambda n^{a(\mu)}e^{(k-l)\alpha_n}\\
& & =D'_\lambda n^{a(\mu)}e^{-t\alpha_n}\leq  D'_\lambda n^c  e^{-t\alpha_n}\leq D'_\lambda S_\lambda  .
\end{eqnarray*}
Hence, for every $\mu\in \ov{B(\lambda, \delta)}$,  we have the inequality
\[
\sup_{n\in\N}\sum_{m=1}^\infty |\tilde{e}^{k,l}_{nm}(\mu)|\leq D'_\lambda S_\lambda ,
\]
that is,  condition (ii) of Lemma 2.1 in \cite{ABR-9} is satisfied for all $\mu\in \ov{B(\lambda, \delta)}$. Moreover, since
$n^{a(\mu) -1 } v_l (n) =  n^{a(\mu)-1}e^{- l \alpha_n}=n^{a(\mu)-1-c}n^c  e^{-t\alpha_n} e^{- k \al_n}\to 0$ for $n\to\infty$
(because  $ S_\lambda = \sup_{n\in\N}n^c  e^{-t\alpha_n}<\infty $, $e^{-k \al_n} \le 1 $,  and $ a(\mu) <  a(\lambda)+\frac{1}{2} \le c +1$),
the inequality  \eqref{eq.stimap}  implies for each fixed  $\mu\in \ov{B(\lambda, \delta)}$ and $m\in\N$ that
\[
\lim_{n\to\infty}\tilde{e}^{k,l}_{nm}(\mu)=0.
\]
Also the condition (i) of Lemma 2.1 in \cite{ABR-9} is satisfied, for all $\mu\in \ov{B(\lambda, \delta)}$.  Accordingly, \cite[Lemma 2.1]{ABR-9}
implies,  for every $\mu\in \ov{B(\lambda, \delta)}$,  that $\tilde{E}_{\mu,k,l}\in \cL(c_0)$ with $\|\tilde{E}_{\mu,k,l}\|_{op}\leq D'_\lambda S_\lambda $,
that is, $\{\tilde{E}_{\mu,k,l}\colon \mu\in\ov{B(\lambda, \delta)}\}$ is equicontinuous in $\cL(c_0)$.
Finally, in view of  \eqref{eq.primo},  we have shown that condition \eqref{eq.condequi} is indeed  satisfied.
\end{proof}

\begin{corollary} \label{C210}
For $\al$ with $\al_n \uparrow \infty $ the following assertions are equivalent.
\begin{enumerate}
  \item [\rm (i)]
$E_\al$   is nuclear.
  \item [\rm (ii)]
$\si (\sC ; E_\al) = \si_{pt} (\sC ; E_\al)$.
  \item  [\rm (iii)]
 $\si (\sC ; E_\al) = \Si$.
\end{enumerate}
\end{corollary}

\begin{proof}
(i)$\Rightarrow$(ii) and (i)$\Rightarrow$(iii) are clear from Proposition \ref{P.spectrum}(i).

(ii)$\Rightarrow$(i). The equality in (ii) together with the fact that $\si_{pt} (\sC ; E_\al) \su \Si$ (see the discussion prior to
Proposition \ref{P.spP}) implies $0 \in \rho (\sC ; E_\al)$. Hence, $E_\al$ is nuclear; see Proposition \ref{P.sp0}.

 (iii)$\Rightarrow$(i).
The equality in (iii)  implies $0 \in \rho (\sC ; E_\al)$ and so $E_\al$ is nuclear (cf. Proposition \ref{P.sp0}).
\end{proof}

Recall that an operator $T \in \cL (X)$, with $X$ a lcHs, is \textit{compact} (resp. \textit{weakly compact}) if there exists a
neighbourhood $U$ of $0$ such that $T (U)$ is a relatively compact (resp. relatively weakly compact) subset of $X$.

\begin{corollary}\label{C211}
Let $\al$ satisfy $\al_n \uparrow \infty $ with $E_\al$ nuclear. Then the Cesàro operator $\sC \in \cL (E_\al)$ is neither compact
nor weakly compact.
\end{corollary}

\begin{proof}
Since $E_\al$ is Montel, there is no distinction between $\sC$ being compact or weakly compact. So, suppose that $\sC$ is
compact. Then $\si (\sC ; E_\al)$ is necessarily a compact set in $\C$, \cite[Theorem 9.10.2]{E}, which contradicts Proposition \ref{P.spectrum}(i).
\end{proof}

The identity $\sC = \Delta \mbox{diag} ((\frac 1 n)) \Delta $ holds in $\cL (\C^\N )$ and all the three operators $\sC , \Delta$
and $\mbox{diag} ((\frac 1 n ))$ are continuous; see the discussion prior to Proposition \ref{P.spP}. For \textit{every}  positive sequence
$\al_n \uparrow \infty $ we also have that $\sC \in \cL (E_\al)$ (cf. Proposition \ref{C-Cont}) and $\mbox{diag} ((\frac 1 n ))
\in \cL (E_\al)$ (because $\mbox{diag} ((\frac 1 n )) \in \cL (c_0 (v_k))$ for every $k \in \N $). If $\Delta $ acts in $E_\al$,
then $\Delta e_n \in E_\al$ for all $n \in  \N$ and so $\si_{pt} (\sC ; E_\al) = \Si $; see \eqref{22}. Accordingly, $E_\al$ is
necessarily nuclear via Proposition \ref{P.spP}. However, this condition alone is not sufficient for the continuity of $\Delta$.

\begin{prop}\label{pr12}
For $\al$ with $\al_n \uparrow \infty $ the following assertions are equivalent.
\begin{enumerate}
  \item [\rm (i)]
The operator $\Delta \in \cL (E_\al)$.
  \item [\rm (ii)]
$\sup_{n \in \N } \frac{n}{\al_n} < \infty $.
\end{enumerate}
\end{prop}

\begin{proof}
 For each $k \in\N $,  the surjective isometric isomorphism $\Phi_k : c_0 (v_k ) \to c_0$ was defined in the proof of
Proposition \ref{P.spectrum}. Because $E_\al = \ind_k c_0 (v_k)$ it follows that $\Delta \in \cL (E_\al)$ if and only if for each
$k \in \N $ there exists $l \in \N $ with $l > k $ such that $\Delta : c_0 (v_k) \to c_0 (v_l)$ is continuous. Moreover, the
continuity of $\Delta : c_0 (v_k) \to c_0 (v_l)$ is equivalent to  continuity of the operator  $D^{k,l} : c_0 \to c_0$, where $D^{k,l} :=
\Phi_l \Delta \Phi_k^{-1}$. Note that $\Phi_l = \mbox{diag} ((v_l (n)))$ and $\Phi_k^{-1} = \mbox{diag} ((\frac{1}{v_k (n)}))$ are diagonal
 matrices and $\Delta = (\Delta_{nm})_{n,m \in \N }$ is a lower triangular matrix, a direct calculation shows that
 $D^{k,l} = (d^{k,l}_{n m })_{n , m \in \N }$ is the lower triangular matrix where, for each $n \in \N $, $d^{k,l}_{nm} = (-1)^{m-1}
 \frac{v_l (n)}{v_k (m)} \binom{n-1}{m-1}$, for $1 \le m < n $ and $d^{k,l}_{nm} =0$ if $m > n $. It follows from \cite[Theorem 4.51-C]{T}
 that a matrix $A = (a_{nm})_{n,m \in \N }$ acts continuously on $c_0$ if and only if the matrix $(|a_{nm}|)_{n,m \in \N }$ does so and hence,
 by the same result in \cite{T}, that $\Delta \in \cL(E_\al)$ if and only if for each $k \in \N $ there exists $l \in \N $ with $l>k$ such
 that the lower triangular matrix $(|d_{nm}^{k,l}|)_{n,m \in \N }$ satisfies both
 \begin{equation}\label{220}
    \lim_{n \to \infty } |d_{nm}^{k,l}| = \lim_{n \to \infty }  \frac{v_l (n)}{v_k (m)} \binom{n-1}{m-1} = 0 , \qquad \forall m \in \N ,
 \end{equation}
 and
 \begin{equation}\label{221}
    \sup_{n \in \N } \sum^\infty_{m=1 }  |d_{nm}^{k,l}| = \sup_{n \in \N } \sum^n_{m=1}  \frac{v_l (n)}{v_k (m)} \binom{n-1}{m-1} < \infty.
 \end{equation}
Actually, \eqref{221} implies \eqref{220}. Indeed, if \eqref{221}  holds, then there exists $L > 0 $ satisfying $v_l (n) \sum^n_{m=1}
\frac{1}{v_k (m)} \binom{n-1}{m-1}  \le L $ for all $n \in \N $ and hence, as $\frac{1}{v_k (m)} = e^{k \al_m} > 1$ for all $m \in \N$,
also $2^{n-1} v_l (n) = v_l (n) \sum^n_{m=1 }  \binom{n-1}{m-1} \le L$ for all $n \in \N $. Then, for fixed $m \in \N $, it follows that
$$
n^{m-1} v_l (n) = \frac{n^{m-1}}{2^{n-1}} \cdot 2^{n-1} v_l (n) \le \frac{L \cdot  n^{m-1}}{2^{n-1}} , \qquad n \in \N .
$$
Since $(\frac{n^{m-1}}{2^{n-1}})_{n \in \N }$ is a null sequence and $ \binom{n-1}{m-1} \simeq n^{m-1}$ for $n \to \infty $ the condition
\eqref{220} follows. So, we have established that the continuity of $\Delta: E_\al \to E_\al$ is \textit{equivalent} to the following

\noindent \textit{Condition} $(\delta)$: For every $k \in \N $ there exists  $l > k$ such that \eqref{221} is satisfied.

(i)$\Rightarrow$(ii). Since Condition $(\delta)$ holds, for the choice $k=1$ there exist $l \in \N $ with $l > 1 $ and $M>1$ such that
$$
2^{n-1} v_l (n ) = v_l (n) \sum^n_{m=1} \binom{n-1}{m-1} \le \sum^n_{m=1} \frac{v_l (n)}{v_1 (m)} \binom{n-1}{m-1} \le M ,
\qquad n \in \N .
$$
Hence, $2^n v_l (n) \le 2 M $ from which it follows that
$$
\exp (n \log (2) - l a_n) \le 2 M = \exp (\log (2M)), \qquad n \in \N .
$$
Rearranging this inequality yields
$$
\frac{n}{\al_n} \le \frac{l}{\log (2)} + \frac{\log (2M)}{\al_n \log (2)} , \qquad n \in \N .
$$
Since $\al_n \uparrow \infty$, it follows that $\sup_{n \in \N } \frac{n}{\al_n} < \infty $.

(ii)$\Rightarrow$(i). Choose $M \in \N $ such that $n \le M \al_n$ for $n \in \N $. In order to verify Condition ($\delta$) fix $k \in \N $. Then
$l := (k+M) \in \N $ and $l > k$. Since $v_k$ is decreasing on $\N $ we have
$$
\sum^n_{m=1} \frac{v_l (n)}{v_k (m)} \binom{n-1}{m-1} \le  \frac{v_l (n)}{v_k (n)} \sum^n_{m=1} \binom{n-1}{m-1} \le 2^n
\frac{v_l (n)}{v_k (n)} , \qquad n \in \N .
$$
Furthermore, for each $n \in \N $, it is also the case that
$$
2^n \frac{v_l (n)}{v_k (n)} = 2^n e^{- \al_n (l-k)} = e^{n \log (2)} e^{- M \al_n} \le e^n e^{-M \al_n} \le 1 .
$$
The previous two sets of inequalities imply \eqref{221} and hence,  Condition $(\delta)$ is satisfied, i.e., $\Delta \in \cL (E_\al)$.
\end{proof}

\begin{remark}\label{R.213} \rm
(i) Clearly $\sup_{n \in \N } \frac{n}{\al_n} < \infty $ implies $E_\al$ is a nuclear space (cf. Proposition \ref{P.sp0}). On the other hand,
the sequence $\al_n := \log (n)$, $n \in \N $, has the property that $E_\al$ is nuclear but, $\Delta \notin \cL (E_\al)$ by Proposition \ref{pr12}.

(ii) The continuity of the operators $\Delta $ and $D $ on $E_\al$ is unrelated. Indeed, consider $\al_n := \sqrt{n}$, for $n \in \N $. Then
$D$ is continuous because $E_\al$ is both nuclear and shift stable (cf. Proposition \ref{pr25}) whereas $\Delta $ is not continuous
(cf. Proposition \ref{pr12}). On the other hand, $\Delta $ is continuous on $E_\al$ for $\al_n := n^n$, $n \in \N $ (via Proposition \ref{pr12}), but
$D $ fails to be continuous on this space; see Remark \ref{r26}.
\end{remark}

We end this section with an application. Consider the space of germs of holomorphic functions at $0$, namely the regular (LB)-space
defined by $H_0 := \ind_k A (\ov{ B (0, \frac 1k)})$. Here, for each $k \in \N $, $A (\ov{B (0, \frac 1k)})$ is the disc algebra consisting of all
holomorphic functions on the open disc $B (0, \frac 1k) \su \C $ which have a continuous extension to its closure $\ov{B (0, \frac 1k)}$: it
is a Banach algebra for the norm
$$
\|f\|_k := \sup_{|z|\le \frac 1 k} |f (z)| = \sup_{|z|= \frac 1k} |f (z)|, \qquad f \in A (\ov{B (0, \textstyle \frac 1k)}).
$$
It is known that the linking maps $A (\ov{B (0, \frac 1k)}) \to A( \ov{B (0, \frac {1}{k+1})})$ for $k \in \N $, which are given by restriction, are
injective and absolutely summing. By Köthe duality theory, $H_0$ is isomorphic to the strong dual of the nuclear Fréchet space $H (\C)$.
In particular, $H_0$ is a (DFN)-space. We refer to \cite[Section 2, Example 5]{B} and \cite[Ch. 5.27, Sections 3,4]{Ko} for further information
concerning spaces of holomorphic germs and their strong duals. Define $\al = (\al_n)$ by $\al_n := n $ for $n \in \N $ in which case
$\lim_{n \to \infty} \frac{\log (n)}{\al_n} =0$. Then $H (\C)$ is isomorphic to the power series  space $\Lambda_\infty^1 (\al)$ of infinite
type, \cite[Example 29.4(2)]{MV}, and its strong dual $E_\al$ is isomorphic to $H_0$. Indeed, a topological isomorphism of $H_0$ onto
$E_\al$ is given by the linear map which sends $f (z)= \sum^\infty_{n=0} a_n z^n$ (an element of $A (\ov{B (0, \frac 1k)})$ for some $k \in \N $)
to $(a_{n-1})_{n\in \N } \in E_\al$. The proof of this (known) fact relies on the following estimates.

(i) If $f \in A (\ov{B (0, \ve)})$ for some $0 < \ve <1 $ (with $f (z) = \sum^\infty_{n=0} a_n z^n$), then the Cauchy estimates for $f$ imply
$|a_n| \le \frac{1}{\ve^n} \max_{|z|= \ve } |f (z)|$ for $n \in \N_0 := \{0\} \cup \N $. Hence, if $f \in A (\ov{B (0, \frac 1k)})$ for some $k \in \N $, then
$$
|a_n|  \le k^n \max_{|z| = \frac 1k} |f (z)| = k^n \|f\|_k , \qquad n \in \N_0 .
$$

(ii) Let $a:= (a_n)_{n \in \N_0} \in \ell_\infty (v_k)$ for some $k \in \N $, where $v_k (n):= \frac{1}{(1+k)^n}$  for $n \in \N_0$,
$k \in \N $; we have taken here $s_k := \log (k+1)$.  Then $|a_n| \le q_k (a) k^n$ for $n \in \N_0$ and each fixed $k \in \N $.
Hence, if $|z| \le \frac{1}{2k}$, then $f (z) = \sum^\infty_{n=0} a_n z^n$ satisfies
$$
|f (z) | \le \sum^\infty_{n=0} |a_n| \cdot |z|^n \le q_k (a) \sum^\infty_{n=0} k^n \frac{1}{(2k)^n} = 2 q_k (a)  .
$$
Accordingly, $f \in A (\ov{B (0, \frac {1}{2k})})$.

The above facts, combined with Proposition \ref{P.spectrum} and Corollary \ref{C211}, yield the following result.

\begin{prop} \label{p214}
The Cèsaro operator $\sC : H_0 \to H_0$ is continuous with spectra
$$
\si (\sC ; H_0) = \si_{pt} (\sC ; H_0) = \Si  \quad \mathrm{and}  \quad \si^* (\sC ; H_0) = \Si_0 .
$$
In particular, $\sC$ is not {\rm(}weakly{\rm)} compact.
\end{prop}

\section{The spectrum of $\sC$ in the non-nuclear case}

The aim of this section is to give a complete description of the spectrum of $\sC \in \cL (E_\al)$ for the case when
$E_\al $ is \textit{not} nuclear. It turns out that $\si (\sC; E_\al)$ and $\si^* (\sC; E_\al)$ are dramatically different to
that when $E_\al$ is nuclear. The following fact, which we record for the sake of explicit reference, is immediate from
\eqref{23} and Propositions \ref{P.spP} and \ref{P.sp0}.

\begin{prop} \label{p31}
For $\al$ with $\al_n \uparrow \infty $ the following assertions are equivalent.
\begin{enumerate}
  \item [\rm (i)]
$E_\al$   is not nuclear.
\item[\rm (ii)]
$\si_{pt} (\sC; E_\al) = \{1\}$.
\item[\rm (iii)]
$0 \in \si (\sC; E_\al )$.
\end{enumerate}
\end{prop}

The following general result will be useful in the sequel. For  each $r > 0$ we adopt the notation $D (r) := \{ \lambda \in \C :
|\lambda- \frac{1}{2r}| < \frac{1}{2r}\}$.

\begin{prop}\label{p32}
Let $\al $ satisfy $\al_n \uparrow \infty $. Then
$$
\Si \su \si (\sC ; E_\al)  \su \ov{D(1)} .
$$
\end{prop}

\begin{proof}
Since $\sC \in \cL (E_\al)$, its dual operator $\sC '$ is defined, continuous on the strong dual $E'_\al = \bigcap_{k \in \N } \ell_1
(\frac{1}{v_k}) = \proj_k \ell_1 (\frac{1}{v_k})$  of $E_\al = \ind_k c_0 (v_k)$ and is given by the formula
$$
\sC ' y := \Big ( \sum^\infty_{j = n} \frac{y_j}{j}\Big)_{n \in \N }, \qquad y = (y_n) \in E'_\al ;
$$
see (3.7) in \cite[p.\,774]{ABR-9}, for example, after noting that $E'_\al \su \ell_1 (\frac{1}{v_1})$. Given $\lambda \in \Si$ there is
$m \in \N $ with $\lambda = \frac 1 m$. Define $u^{(m)}$ by $u^{(m)}_n := \prod^{n-1}_{k=1} (1- \frac{1}{\lambda k})$ for $1<n\le m$
(with $u^{(m)}_1 := 1$) and $u^{(m)}_n := 0$ for $n  >m$. It is routine to verify that $u^{(m)} \in E'_\al$ (as $u^{(m)} \in \vp$) and
$\sC' u^{(m)} = \frac 1 m u^{(m)}$, i.e., $\lambda \in \si_{pt} (\sC'; E'_\al)$. It follows that $\lambda \in  \si (\sC ;E_\al)$. Indeed, if not,
then $\lambda \in \rho (\sC ;E_\al)$ and so $(\sC - \lambda I) (E_\al) = E_\al$. This implies, for each $z \in E_\al$ that there exists $x \in E_\al$
satisfying $(\sC - \lambda I ) x = z $. Hence,
$$
\la z, u^{(m)}\ra = \la (\sC - \lambda I) x , u^{(m)}\ra  = \la x, (\sC' - \lambda I) u^{(m)}\ra =0 ,
$$
that is, $\la z, u^{(m)}\ra =0$ for all $z \in E_\al$. Since $u^{(m)} \ne 0$, this is a contradiction. So, $\lambda \in \si (\sC ; E_\al) $. This
establishes that $\Si \su \si (\sC ; E_\al)$.

According to Lemma \ref{L28} we see that $\si (\sC_k ; c_0 (v_k)) \su \ov{D(1)}$ for all $k \in \N $, where $\sC_k : c_0 (v_k) \to c_0 (v_k)$
is the restriction of $\sC \in \cL (\C^\N)$. Hence,
$$
\bigcap_{m \in \N } \Big(\bigcup_{k=m}^\infty  \si (\sC_k ; c_0 (v_k))\Big) \su \ov{D (1)}
$$
and so   $\si (\sC; E_\al) \su \ov{D(1)}$; see Lemma \ref{Le-spectrumNV} in the Appendix.
\end{proof}

The following result identifies a large part of $\si (\sC ; E_\al)$.

\begin{prop}\label{p33}
Let $\al $ satisfy $\al_n \uparrow \infty $ and such that $E_\al$ is not nuclear. Then
$$
\{0,1\} \cup D (1) \su \si (\sC; E_\al) \su \ov{D(1)} .
$$
\end{prop}

\begin{proof}
It follows from Propositions \ref{p31} and \ref{p32} that $\Si_0 \su \si (\sC; E_\al) \su \ov{D(1)}$. So, it remains to verify that $(D (1)
\s \Si ) \su \si (\sC; E_\al) $. This is achieved via a contradiction argument.

Let $\lambda \in D (1) \s \Si $ and suppose that $\lambda \in \rho  (\sC; E_\al) $. Note that $\beta := \mbox{Re} (\frac 1 \lambda ) > 1$.
Since $(\sC - \lambda I )^{-1} : E_\al \to E_\al$  is continuous, for $k = 1$ there exists $l \in \N $ with $l > 1 $ such that $(\sC - \lambda I)^{-1} :
c_0 (v_1) \to c_0 (v_l)$ is continuous. In the notation of the proof of Proposition \ref{P.spectrum} it follows that the linear map
$\wt{E}_{\lambda,1 ,l} : c_0 \to c_0$ is continuous, where $\wt{E}_{\lambda,1 ,l} = (\tilde{e}_{n m}^{1, l} (\lambda ))_{n, m \in \N }$ is the lower
triangular matrix given by
\begin{equation}\label{n31}
    \tilde{e}_{n m}^{1, l} (\lambda ) = \frac{v_l (n)}{v_1 (m)} e_{n m} (\lambda ), \qquad \forall n  \ge 2 , \quad 1 \le m < n ,
\end{equation}
and $\tilde{e}_{n m}^{1, l} (\lambda ) =0$ otherwise. Here $e_{n,m} (\lambda ) = \frac{1}{n \prod^n_{k=m} (1- \frac{1}{ \lambda k})}$  if $1 \le m <n$
and $e_{nm} (\lambda ) =0$ if $m\ge n$. According to the inequality (3.10) in \cite[p.\,776]{ABR-9}, there exist positive constants
$c,d$ such that
\begin{equation}\label{n32}
    \frac{c}{n^{1- \beta }}  \le |e_{n1} (\lambda )| \le \frac{d}{n^{1- \beta }} , \qquad n \ge 2 .
\end{equation}
Since  $\wt{E}_{\lambda, 1 , l } \in \cL (c_0)$, a well known criterion, \cite[Lemma 2.1]{ABR-9}, \cite[Theorem 4.51-C]{T}, implies that necessarily
\begin{equation}\label{n33}
    \lim_{n \to \infty } \tilde{e}_{n m}^{1, l} (\lambda ) =0, \quad m \in \N.
\end{equation}
It now follows from \eqref{n31}, the left-inequality in \eqref{n32}, and \eqref{n33} with $m =1$, that
$$
\lim_{n \to \infty } n^{\beta-1} e^{- l \al_n} = \lim_{n \to \infty } n^{\beta-1} v_l (n) =0  .
$$
Since   $\beta > 1$, it follows from Lemma \ref{L.s1} that $\sup_{n \in \N } \frac{\log (n)}{\al_n} < \infty $ which contradicts the non-nuclearity of
$E_\al$ (cf. Proposition \ref{P.spP}). Hence, \textit{no} $\lambda \in D (1) \s \Si$ exists with $\lambda \in \rho (\sC; E_\al)$.
\end{proof}

We now come to the main result of this section.
\begin{prop} \label{pr34}
Let $\al$ satisfy $\al_n \uparrow \infty $ and such that $E_\al$ is not nuclear.
\begin{enumerate}
  \item [\rm (i)]
 If $\sup_{n \in \N }  \frac{\log (\log (n))}{\al_n} < \infty $, then
 $$
\si (\sC; E_\al) = \{0,1 \} \cup D(1) \ \mbox{ and } \ \; \si^* (\sC; E_\al) = \ov{D(1)} .
$$
  \item [\rm (ii)]
If   $\sup_{n \in \N }  \frac{\log (\log (n))}{\al_n} = \infty $, then
$$
\si (\sC; E_\al) = \ov{D(1)} = \si^* (\sC; E_\al)  .
$$
\end{enumerate}
\end{prop}

\begin{proof}
In the notation of the proof of Proposition \ref{P.spectrum}, for each $\lambda \in \C \s \Si_0$ the inverse operator
 $(\sC - \lambda I)^{-1} \in \cL (\C^N)$   satisfies
$$
(\sC - \lambda I)^{-1} = D_\lambda - \frac{1}{\lambda^2} E_\lambda ;
$$
see \eqref{e:dec1}. It is also argued there (as a consequence of the fact that the diagonal in $D_\lambda$ is a bounded sequence)
that $(\sC - \lambda I)^{-1} : E_\al \to E_\al$ is continuous if and only if $E_\lambda \in \cL (E_\al) $; the nuclearity of $E_\al$ is not
used for this part of the argument. Moreover, since $E_\al$ is an inductive limit, general theory yields that $E_\lambda \in \cL (E_\al)$
if and only if for each $k \in \N $ there exists $l \in \N $ with $l > k$ such that $E_\lambda : c_0 (v_k) \to c_0 (v_l)$ is continuous. With
$\wt{E}_{\lambda,k ,l} = (\tilde{e}_{n m}^{k, l} (\lambda ))_{n, m \in \N }$, where $\tilde{e}_{n m}^{k, l} (\lambda ) := \frac{v_l (n)}{v_k (m)}
e_{nm} (\lambda)$ for $n , m \in \N $, it follows via the argument used in Case (ii) of the proof of Proposition \ref{P.spectrum} (see also
the proof of Proposition \ref{p33}, where $k=1$ can be replaced by an arbitrary $k \in \N $) that $E_\lambda : c_0 (v_k) \to c_0 (v_l)$ is
continuous if and only if $\wt{E}_{\lambda, k , l } : c_0 \to c_0$ is continuous. Via  \cite[Theorem 4.51-C]{T} this
is equivalent to both of the following conditions being satisfied:
\begin{equation}\label{n34}
    \lim_{n \to \infty } |\tilde{e}_{n m}^{k, l} (\lambda )| = \lim_{n \to \infty } \frac{v_l (n)}{v_k (m)} |e_{nm} (\lambda) | =0 , \qquad \forall  m \in \N ,
\end{equation}
and
\begin{equation}\label{n35}
 \sup_{n \in \N }    \sum^\infty_{m=1} \frac{v_l (n)}{v_k (m)} |e_{nm} (\lambda) |= \sup_{n \in \N } \sum^{n-1}_{m=1}
 \frac{v_l (n)}{v_k (m)} |e_{nm} (\lambda) | < \infty .
\end{equation}

Next, if $\lambda \notin \{0,1\}$ belongs to the boundary $\partial D (1)$ of $D(1)$, then $\beta := \mbox{Re} (\frac 1 \lambda ) =1$ and
$\lambda \notin \Si_0$. Accordingly, Lemma 3.3 of \cite{ABR-9} ensures the existence of positive constants $c,d$ such that
$c \le |e_{n1} (\lambda)| \le d $ for all $n \in \N $ and
\begin{equation}\label{n36}
    \frac c m \le |e_{nm} (\lambda )| \le \frac d m , \qquad \forall n \in \N , \quad 2 \le m <  n .
\end{equation}
In order to deduce \eqref{n36} from \cite[Lemma 3.3]{ABR-9} we have used the formula
$$
 |e_{nm} (\lambda )| = \frac{1}{(m-1)} \cdot \frac{(m-1) \prod^{m-1}_{k=1}|1- \frac{1}{\lambda k}|}{n \prod^n_{k=1}|1- \frac{1}{\lambda k }|},
 \qquad \forall n \in \N , \: \ 2 \le m <  n  .
$$
Henceforth we use $v_r (n) := e^{-r \al_n}$ for all $r, n \in \N $.
Note that \eqref{n34} is satisfied for every $ \lambda \in \partial D(1) \s \{0,1\}$. Indeed, for fixed $m \in \N$, we have via \eqref{n36} that
$$
 \frac{v_l (n)}{v_k (m)} |e_{nm} (\lambda) | \le \frac{de^{k \al_m}}{m e^{l\al_n}} \le \frac{d'}{e^{l\al_n}}, \qquad n \in \N ,
$$
from which \eqref{n34} is clear.

(i) Since $\sup_{n \in \N } \frac{\log (\log (n))}{\al_n} < \infty $, there exists $M \in \N $ such that $\log (\log (n)) \le M \al_n$, equivalently
$ \log (n) \le e^{M \al_n}$ for $n \in \N $. Fix $\lambda \in \partial D(1) \s \{0,1\}$; in particular,
$\lambda \notin \Si_0$. Given $k \in \N $ define $l := k+M$. Then, for every $n \ge 2 $, it follows from \eqref{28}, \eqref{n36} and
$ (l-k)=M$ that
\begin{eqnarray*}
&& \sum^{n-1}_{m=1} \frac{v_l (n)}{v_k (m)} |e_{nm} (\lambda) | \le \frac{d}{e^{l \al_n}} \sum^{n-1}_{m=1} \frac{e^{k \al_m}}{m} \le
 \frac{d e^{k \al_n}}{e^{l \al_n}} \sum^{n-1}_{m=1} \frac 1m \\
 && \qquad \le \frac{1+ \log (n)}{e^{M \al_n}} = e^{- M \al_n} + \frac{\log (n)}{e^{M \al_n}} \le 2 .
\end{eqnarray*}
Accordingly, \eqref{n35} is satisfied. Since \eqref{n34} holds, we conclude that $\wt{E}_{\lambda , k, l} : c_0 \to c_0$ is
continuous, equivalently that $ (\sC - \lambda I)^{-1} \in \cL (E_\al)$. It follows that  $\partial D (1) \s \{0,1\} \su \rho (\sC; E_\al)$ and so
$\si (\sC; E_\al)= \{0,1\} \cup D (1)$; see Proposition \ref{p33}.

It was shown in the proof of Proposition \ref{p32} that $\bigcup^\infty_{k=1} \si (\sC_k ; c_0 (v_k)) \su \ov{D(1)}$. Since $\si (\sC; E_\al) =
\{0,1\} \cup D (1)$, we have $\ov{\si (\sC; E_\al)} = \ov{D(1)}$ and so $\bigcup^\infty_{k=1} \si (\sC_k ; c_0 (v_k)) \su \ov{\si (\sC; E_\al)}$.
It follows from Lemma \ref{Le-spectrumNV}(iii) in the Appendix that $\si^* (\sC; E_\al)= \ov{D(1)}$.

(ii) Fix $\lambda \in \partial D (1) \s \{0,1\}$. Observe first, for $k=1$ and $l \in \N $ \textit{arbitrary}, that it follows from \eqref{28} and
\eqref{n36} that
\begin{equation}\label{n37}
  \sum^{n-1}_{m=1} \frac{v_l (n)}{v_k (m)} |e_{nm} (\lambda) | \ge \frac{c}{e^{l \al_n}} \sum^{n-1}_{m=1} \frac{e^{\al_m}}{m}   \ge
 \frac{c e^{ \al_1}}{e^{l \al_n}} \sum^{n-1}_{m=1} \frac 1m  \ge \frac{c \log (n)}{e^{l \al_n}} ,
\end{equation}
for all $n \ge 2 $. Suppose now that $\lambda \in \rho (\sC; E_\al)$. Then for $k=1$ there exists $l \in \N $ with $l > 1$ such that
\eqref{n35} is satisfied. It then follows from \eqref{n37} that $ \sup_{n \in \N } \frac{\log (n)}{e^{l \al_n}} < \infty $. So, there exists
$K> 1$ such that $\log (n) \le K e^{l \al_n}$, equivalently that
$$
\log (\log (n)) \le l \al_n + \log (K) , \qquad n \ge 3  .
$$
A rearrangement yields $\frac{\log (\log (n))}{\al_n} \le l + \frac{\log (K)}{\al_n}$ for $n \ge 3 $, and so $\sup_{n \in \N } \frac{\log (\log (n))}{\al_n}
< \infty $; contradiction! So, \textit{no}  $\lambda \in \partial D(1) \s \{0,1\}$ exists which satisfies $\lambda \in \rho (\sC; E_\al)  $, i.e.,
$\partial D(1) \s \{0,1\} \su \si (\sC; E_\al) $. It now follows from Proposition \ref{p33} that $\si (\sC; E_\al) = \ov{D(1)}$.

It was observed in the proof of part (i) that $\bigcup^\infty_{k=1} \si (\sC_k; c_0 (v_k)) \su \ov{D(1)} $. Since $\ov{D(1)} = \si (\sC; E_\al) =
\ov{\si (\sC; E_\al)} $, it again follows from Lemma \ref{Le-spectrumNV}(iii) in the Appendix that $\si^* (\sC; E_\al) = \si (\sC; E_\al)  $.
\end{proof}

\begin{remark} \rm

(i) Let $\al $ satisfy $\al_n \uparrow \infty  $. Then $\si (\sC ; E_\al)$ is a compact subset of $\C$ if and only if $\sup_{n \in \N }
\frac{\log (\log (n))}{\al_n} = \infty$. This follows from Corollary \ref{C210}, Proposition \ref{pr34} and the fact that the
 condition $\sup_{n \in \N } \frac{\log (\log (n))}{\al_n} = \infty $ implies
$\sup_{n \in \N } \frac{\log (n)}{\al_n} = \infty $, i.e., $E_\al$ is automatically non-nuclear.

(ii) The sequence $\al_n := \log (\log (n))$ for
$n \ge 3^3 > e^e $ (with $1< \al_1 < \ldots < \al_{26} < \log (\log(3^3))$ arbitrary) satisfies $1< \al_n \uparrow \infty $ with $E_\al$ not
nuclear and $\sup_{n \in \N } \frac{\log (\log (n))}{\al_n} < \infty $. Proposition \ref{pr34}(i) shows that $\si (\sC; E_\al)= \{0,1\} \cup D(1)$.
On the other hand, the sequence $\al_n : = \log (\log (\log(n)))$ for $n \ge 3^{27} > e^{e^e}$ (with $1< \al_1 < \ldots < \al_{3^{27}-1} <
\log (\log (\log (3^{27})))$ arbitrary) satisfies $1 < \al_n \uparrow \infty $ with $E_\al$ not nuclear and
$\sup_{n \in \N } \frac{\log (\log (n))}{\al_n} = \infty $. In this case Proposition \ref{pr34}(ii) reveals that $\si (\sC; E_\al) = \ov{D(1)}$.
\end{remark}

\section{ Mean ergodicity of the Ces\`aro operator. }

An operator $T \in \cL (X)$, with $X$ a lcHs, is \textit{power bounded }if $\{T^n\}^\infty_{n=1}$ is an equicontinuous subset of $\cL (X)$.
Given $T \in \cL (X)$, the averages
$$
T_{[n]} := \frac 1 n \sum^n_{m=1} T^m , \qquad n  \in \N ,
$$
are called the Cesàro means of $T$. The operator $T$ is said to be \textit{mean ergodic} (resp. \textit{uniformly mean ergodic}) if
$\{T_{[n]}\}^\infty_{n=1}$ is a convergent sequence in $\cL_s (X)$ (resp., in $\cL_b (X)$). A relevant text for mean ergodic operators is \cite{K}.

\begin{prop}\label{p.uniforE}
Let $ \al_n \uparrow \infty $. The  Ces\`aro operator $\sC \in \cL (E_\al )$ is power bounded and uniformly mean ergodic.  In particular,
\begin{equation}\label{31}
    E_\alpha=\Ker (I-\sC)\oplus \ov{(I-\sC)(E_\alpha)}
\end{equation}
with
\begin{equation}\label{32}
    \Ker (I-\sC)=\{\boldsymbol{1}\}  \mbox{ and }  \ov{(I-\sC)(E_\alpha)} =\{x\in E_\alpha\colon x_1=0\}=\ov{{\rm span}\{e_n\}_{n\geq 2}}  .
\end{equation}
\end{prop}

\begin{proof} Since each  weight $v_k$ for $k \in \N $ is decreasing, it is known that    $\sC\in\cL(c_0(v_k))$
and $q_k(\sC x)\leq q_k(x)$ for all $x\in c_0(v_k)$, \cite[Corollary 2.3(i)]{ABR-9}. It follows, via \eqref{21},  for every $k  \in\N$ that
$$
 q_k(\sC^m x)\leq q_k(x),  \qquad  \forall  x\in c_0(v_k), \;  m \in \N .
$$
Accordingly, for each $k \in \N $,   \eqref{eq.Eq} is satisfied with $l : =k$ and $D=1 $. Then Lemma \ref{Lereg} in the Appendix implies  that
$\cH:=\{\sC^m\colon m\in\N\} \su \cL (E_\al )$ is equicontinuous, i.e.,  the Ces\`aro operator $\sC$ is  power bounded in  $E_\alpha$.
Since $E_\alpha$ is  Montel,  it follows via \cite[Proposition 2.8]{ABR-0} that the Ces\`aro operator $\sC$ is uniformly
mean ergodic in $E_\alpha$ and hence, \eqref{31} is also satisfied, \cite[Theorem 2.4]{ABR-0}. The facts that each $x \in E_\al$ belongs to
$c_0 (v_k)$ for some $k \in \N $, that the inclusion $c_0 (v_k) \su E_\al$ is continuous and that the canonical vectors $e_n :=
(\delta_{nk})_{k \in \N }$, for $n \in \N $, form a Schauder basis in $c_0 (v_k)$ implies  $\{e_n : n \in \N \}$ is a Schauder basis for
$E_\al$.
The proof of the identities in \eqref{32}  now follow by  applying the same (algebraic)  arguments as used in  the proof of \cite[Proposition 4.1]{ABR-7}.
\end{proof}

\begin{prop}\label{p.iterE}
Let $ \al_n \uparrow \infty $. The sequence $\{\sC^m\}_{m\in\N}$ converges in $\cL_b (E_\alpha)$ to the projection onto
$\mathrm{span} \{\boldsymbol{1}\}  $ along $\ov{(I- \sC) (E_\al)}$.
\end{prop}

\begin{proof}
 Using Proposition \ref{p.uniforE} we proceed as in the proof of the analogous result when $\sC$ acts in the Frèchet space
  $\Lambda_0(\alpha)$, \cite[Proposition 3.2]{ASM}. Indeed, for each $x\in E_\alpha$, we have that $x=y+z$ with
  $y\in \Ker (I-\sC)={\rm span}\{\boldsymbol{1}\}$ and $z\in \ov{(I-\sC)(E_\alpha)}=\ov{{\rm span}\{e_n\}_{n\geq 2}}$.
  So, for each $m\in\N$ we have $\sC^mx=\sC^m y+\sC^mz$, with $\sC^m y=y\to y$ in $E_\alpha$ as $m\to\infty$.
The   claim is  that the sequence $\{\sC^mz\}_{m\in\N}$ is also convergent in $E_\alpha$.
Indeed, proceeding as in the proof of Proposition 3.2 of \cite{ASM}  one shows,    for each  $r\geq 2$ and $m,\ n\in\N$,   that
$|(\sC^me_r)(n)|\leq \frac{1}{r-1}a_m$, where $(a_m)_{m \in \N }$ is a sequence of positive numbers satisfying $\lim_{m \to \infty } a_m = 0$. Since
$  v_1(n)|(\sC^me_r)(n)|\leq \frac{v_1(n)}{r-1}a_m$,   for each $r\geq 2$ and $ n, m\in\N$, with $1 \ge v_1 (1) \ge v_1 (n)$ for all $n \in \N$ it follows that
$q_1(\sC^me_r)\leq \frac{1}{r-1}a_m$.  We deduce,  for each $r\geq 2$,  that $\sC^me_r\to 0$ in $c_0(v_1)$ and hence, also in $E_\alpha$ as $m\to\infty$.
Since  $\{\sC^m\}_{m\in\N}  \su \cL (E_\al )$ is equicontinuous  and (by \eqref{32}) the linear span of $\{e_n\}_{n\geq 2}$ is dense in $\ov{(I-\sC)(E_\alpha)}$,
it follows that $\sC^mz\to 0$ in $E_\alpha$ as $m\to\infty$ for each $z\in \ov{(I-\sC)(E_\alpha)}$.   So, it has been shown that
$\sC^mx=\sC^my+\sC^mz\to y$ in  $E_\alpha$ as $m\to\infty$, for each    $x\in E_\alpha$, i.e.,  $\{\sC^m\}_{m\in\N}$ converges in $\cL_s(E_\alpha)$.
Since  $E_\alpha$ is a Montel space,  $\{\sC^m\}_{m \in \N }$ also converges  in $\cL_b(E_\alpha)$.
\end{proof}

\begin{prop}\label{p.rangeE}
Let $ \al_n \uparrow \infty $ with  $E_\alpha$ nuclear. Then the range  $(I-\sC)^m(E_\alpha)$ is a closed subspace of $E_\al $ for
each $m \in \N $.
\end{prop}

\begin{proof}
Consider first $m=1$. Set $X(\alpha):=\{x\in E_\alpha\colon x_1=0\}$. The claim is that
\begin{equation}\label{33}
    (I-\sC)(E_\alpha) = (I-\sC) (X(\alpha)) .
\end{equation}
First recall that each sequence $v_k$, for $k \in \N $,  is strictly positive and  decreasing  with $v_k \in c_0 $ and so
$\ov{(I-\sC)(c_0(v_k))}=\{x\in c_0(v_k)\colon x_1=0\}=:X_k$ and $(I-\sC)(X_k)=(I-\sC)(c_0(v_k))$,
  \cite[Lemmas 4.1 and 4.5]{ABR-9}.  Now,   if $x\in X(\alpha)$, then  $x\in X_k$ for some $k\in\N$ and hence,
$$
(I-\sC)x\in (I-\sC)(X_k)=(I-\sC)(c_0(v_k))\su (I-\sC)(E_\alpha).
$$
This establishes one inclusion in \eqref{33}. For the reverse inclusion
let $x\in E_\alpha$. Then $x\in c_0(v_k)$ for some $k\in\N$ and hence, $(I-\sC)x\in (I-\sC)(c_0(v_k))=(I-\sC)(X_k)\su (I-\sC)(X(\alpha))$.
Thus,  the reverse inclusion in \eqref{33} is also valid.

 Because of \eqref{33} and the containment
 $(I-\sC)(E_\alpha)\su  \ov{(I-\sC) (E_\al)} =  X(\alpha)$, which is immediate from Proposition \ref{p.uniforE},  to show that
  $(I-\sC)(E_\alpha)$ is closed in $E_\alpha$ it suffices to show that the continuous linear restriction operator
  $(I-\sC)|_{X(\alpha)}\colon X_\alpha\to X_\alpha$ is bijective, actually surjective.   Indeed, if $(I-\sC)(X(\alpha))=X(\alpha)$,
then $(I-\sC)(E_\alpha)=X(\alpha)$ by \eqref{33} and hence, $(I-\sC)(E_\alpha)$ is a closed subspace of $E_\alpha$.

To establish that $(I- \sC)|_{X_\al}$ is bijective we require the identity $(X(\alpha), \tau)=\ind_k X_k $, where $\tau$ is the relative
topology in $X(\alpha)$ induced from $E_\al$. This identity follows from the general fact that if $ (E, \tilde{\tau}) = \ind_n E_n$ is a
(LB)-space and $F \su E $ is a closed subspace with finite codimension, then $(F, \tilde{\tau}|_F) = \ind_n (F \cap E_n)$ is also a
(LB)-space, \cite[Lemma 6.3.1]{PB}.
Actually,  setting $\tilde{v}_k(n):=v_k(n+1)$ for all $k,\ n\in\N$, we have that $X(\alpha)$ is topologically  isomorphic to $E(\tilde{\alpha}):=\ind_k c_0(\tilde{v}_k)$.
Indeed, the left-shift operator $S\colon X(\alpha)\to E(\tilde{\alpha})$ given by $S(x) : =(x_2,x_3,\ldots)$ for $x=(x_n)_{n\in\N}\in X(\alpha)$
is such an isomorphism (because,  for each $k\in\N$,  the left shift operator $S\colon X_k\to c_0(v_k)$ is a surjective isometry). Consider now
the operator $A:=S\circ (I-\sC)|_{X(\alpha)}\circ S^{-1}\in \cL(E(\tilde{\alpha}))$.  The claim is that $A$ is bijective with $A^{-1}\in \cL(E(\tilde{\alpha}))$.

To establish the above  claim   observe, when interpreted to be acting  in the space $\C^\N$, that  the operator $A\colon \C^\N\to \C^\N$
is bijective (which is a routine verification) and its inverse $B:=A^{-1}\colon \C^\N\to \C^\N$ is determined by the lower triangular matrix
$B=(b_{nm})_{n, m\in \N }$ with entries given as follows: for each $n\in\N$ we have $b_{nm}=0$ if $m>n$, $b_{nm}=\frac{n+1}{n}$ if $m=n$ and
$b_{nm}=\frac{1}{m}$ if $1\leq m<n$.  To show   that $B$ is also the inverse of $A$ acting on $E(\tilde{\alpha})$, we only need to verify   that
$B\in \cL(E(\tilde{\alpha}))$. To establish this  it suffices to show,   for each $k\in\N$, that  there exists $l\geq k$ such that
$\Phi_{\tilde{v}_l}\circ B\circ \Phi^{-1}_{\tilde{v}_k}\in \cL(c_0)$, where for each $h\in\N$ the operator $\Phi_{\tilde{v}_h}\colon c_0(\tilde{v}_h)\to c_0$
given by $ \Phi_{\tilde{v}_h}(x)=(\tilde{v}_h(n+1)x_n)$ for $x\in c_0 (\tilde{v}_h)$ is a surjective  isometry. To  this end, given $k\in\N$
set  $l:=k+1$, say. Then  the lower triangular matrix corresponding to $\Phi_{\tilde{v}_l}\circ B\circ \Phi^{-1}_{\tilde{v}_k}$ is given by
$D:=(\frac{v_l(n+1)}{v_k(m+1)}b_{nm})_{n,m \in \N }$. Moreover, for each fixed $m \in \N $, we have
\[
\lim_{n\to\infty}\frac{v_l(n+1)}{v_k(m+1)}b_{nm}=\frac{1}{mv_k(m+1)}\lim_{n\to\infty}v_l(n+1)=0
\]
and, for each $ n \in \N$, that
\begin{eqnarray*}
& &\sum^\infty_{m=1}\frac{v_l(n+1)}{v_k(m+1)}b_{nm}=\frac{(n+1)}{n}\frac{v_l(n+1)}{v_k(n+1)}+v_l(n+1)\sum_{m=1}^{n-1}\frac{1}{mv_k(m+1)}\\
& &\qquad \leq 2+(s_l)^{-\alpha_{n+1}}\sum_{m=1}^{n-1}\frac{s_k^{\alpha_{m+1}}}{m}\leq 2+\left(\frac{s_k}{s_l}\right)^{\alpha_{n+1}}\sum_{m=1}^{n-1}\frac{1}{m}\\
& &\qquad \leq 2+\left(\frac{s_k}{s_l}\right)^{\alpha_{n+1}}(1+\log (n))  \le 2 + 2 a^{\al_{n+1}} \log (n+1),
\end{eqnarray*}
where $a : = \frac{s_k}{s_l} \in (0,1)$. Since $E_\al $ is nuclear, there exists $M \ge 1 $ such that $\log (n) \le M \al_n$ for all $n \in \N $
and hence, $a^{\al_n} \log (n) \le  M \al_n a^{\al_n}$ for $n \in \N $. Since $f (x):= x a^x$, for $x \in (0, \infty )$, satisfies $f' (x) < 0 $ for
$x >  \frac{1}{\log (\frac 1 a )}$, the function $f $ is decreasing on $( \frac{1}{\log (\frac 1 a )},  \infty)$ which implies  $\sup_{n \in \N } a^{\al_n} \log (n) < \infty$,
i.e., $\sum^\infty_{m = 1 } \frac{v_l (n+1)}{v_k (m+1)} < \infty $ for each $n \in \N $.  Thus,   both
 the conditions (i), (ii) of \cite[Lemma 2.1]{ABR-9} are satisfied. Accordingly, $\Phi_{\tilde{v}_l}\circ B\circ \Phi^{-1}_{\tilde{v}_k}\in \cL(c_0)$.
 The proof that $(I - \sC ) (E_\al)$ is closed is thereby complete.

Since $(I- \sC) (E_\al)$ is closed,  \eqref{31} implies
$E_\al = \Ker (I -\sC) \oplus (I - \sC) (E_\al)$. The proof of (2) $\Rightarrow$ (5) in Remark 3.6 of \cite{ABR-7} then shows that $(I-\sC)^m (E_\al)$ is
closed in $E_\al $ for all $m \in \N $.
\end{proof}

An operator $T \in \cL (X)$, with $X$ a separable lcHs, is called \textit{hypercyclic} if there exists $x \in X$ such that the orbit
$\{T^n x : n \in \N_0\}$ is dense in $X$. If, for some $z \in X$ the projective orbit $\{\lambda T^n z: n \in \N_0, \; \lambda \in \C\}$ is
dense in $X$, then $T$ is called \textit{supercyclic}. Clearly, hypercyclicity implies supercyclicity.

\begin{prop} \label{p34}
Let $\al$ satisfy $\al_n \uparrow \infty $. Then $\sC \in \cL (E_\al)$ is not supercyclic and hence, also not hypercyclic.
\end{prop}

\begin{proof}
It is known that $\sC$ is \textit{not} supercyclic in $\C^\N $, \cite[Proposition 4.3]{ABR-1}. Since $E_\al$ is dense (as it contains $\vp$)
and continuously included in $\C^\N$, the supercyclicity of $\sC$ in any one of the spaces $E_\al$ would imply that $\sC \in \cL (\C^\N)$
is supercyclic.
\end{proof}

\section{Appendix}

In this section we elaborate  on the point raised in Section 1 that the behaviour of the Cesàro operator on the
strong dual $(\Lambda^1_0 (\al))'$ of power series  spaces $\Lambda^1_0 (\al)$ of \textit{finite type}, is not so
relevant in relation to continuity. It turns out that $\sC$ \textit{fails} to act in $(\Lambda^1_0 (\al))'$ for \textit{every} $\al$
with $\al_n \uparrow \infty $ such that $(\Lambda^1_0 (\al))'$ is nuclear. Moreover, there exist $\al_n \uparrow \infty $
such that $(\Lambda^1_0 (\al))'$ is not nuclear and $\sC \in \cL ((\Lambda^1_0 (\al))')$ (cf. Example \ref{ex52}) as well as
other $\al_n \uparrow \infty $ such that $(\Lambda^1_0 (\al))'$ is not nuclear and $\sC \notin \cL ((\Lambda^1_0 (\al))')$;
see Example \ref{ex53}.

In order to be able to formulate the above claims more precisely, let $(v_k)_{k \in \N }$ be a sequence of functions
$v_k : \N \to (0, \infty )$ satisfying $v_k (n) \uparrow_n \infty $, for each  $k \in \N $, with $v_k \ge v_{k+1}$ pointwise on $\N$
and $\lim_{n \to \infty } \frac{v_{k+1} (n)}{v_k (n)} =0$  for all $k \in \N $. Then $\ell_\infty (v_k) \su c_0 (v_{k+1})$ continuously
for each $k \in \N $ and so
$$
k_0 (V) := \ind_k c_0 (v_k) = \ind_k \ell_\infty (v_k)  .
$$
In the notation of Köthe echelon spaces $\lambda_1 (\frac 1 v ) := \proj_k \ell_1(\frac{1}{v_k}) $ is a Fréchet-Schwartz
space whose strong dual space, i.e., the co-echelon space $(\lambda_1 (\frac 1 v ))'_\beta = \ind_k \ell_\infty (v_k) = k_0 (V)$,
is a (DFS)-space. It is known that the regular (LB)-space $k_0 (V)$ is nuclear if and only if the Fréchet-Schwartz space
$\lambda_1 (\frac 1 v )$ is nuclear if and only if the Grothendieck-Pietsch criterion is satisfied: for every $k \in \N $ there exists
$l \in \N $ with $l > k$ such that the sequence $(\frac{v_l (n)}{v_k (n)})_{n \in \N } \in \ell_1$, \cite[Section 21.6]{Ja}. In case $v_k (n) := e^{\al_n /k}$,
for $k, n \in \N $, with $\al_n \uparrow \infty $, then $k_0 (V)$ is the strong dual of the finite type power series  space
(of order $1$) $\Lambda^1_0 (\al) := \proj_k \ell_1 (\frac{ 1}{ v_k} )$. This Fréchet space is nuclear if and only if $\lim_{n \to \infty }
\frac{\log (n)}{\al_n} =0$, \cite[Proposition 29.6]{MV}. Whenever this nuclearity condition is satisfied we have $\Lambda^1_0 (\al) =
\proj_j c_0 (\frac{1}{v_k})$ which is precisely the power series   space $\Lambda_0 (\al)$ in which the operator $\sC$ was investigated in \cite{ASM}.

For the rest of this section, whenever $\al_n \uparrow \infty $ we only consider the weights $v_k (n) := e^{\al_n/k}$ for $k, n \in \N $.

\begin{prop}
Let the sequence $\al_n$ satisfy $\al_n \uparrow \infty $ and $\lim_{n \to \infty } \frac{\log (n)}{\al_n} =0$. Then the Cesàro operator
$\sC$ does not act in $k_0 (V) = \ind_k c_0 (v_k)$.
\end{prop}

\begin{proof}
Since $\lim_{n \to \infty } \frac{\log (n)}{\al_n} =0$, it follows from Lemma 2.2 of \cite{ASM} that $\lim_{n \to \infty } n^t e^{-\al_n} =0$ for
each $t \in \N $, which implies $\lim_{n \to \infty } n e^{- \al_n /l} =0$ for each $l \in \N $. In particular,
\begin{equation}\label{51}
\sup_{n \in \N } \frac{e^{\al_n /l}}{n} = \infty , \qquad \forall l \in \N .
\end{equation}

Suppose that $\sC \in \cL (k_0 (V))$, i.e., for every $k \in \N $ there exists $l \in \N $ with $l > k $ such that $\sC : c_0 (v_k) \to c_0 (v_l)$
is continuous. Then, for $k := 1$ there exists $l_1 > 1 $ such that $\sC : c_0 (v_1) \to c_0 (v_{l_1})$ is continuous, equivalently
\begin{equation}\label{52}
    M:= \sup_{n \in \N } \frac{v_{l_1}(n)}{n} \sum^n_{m=1} \frac{1}{v_1 (m)} < \infty ,
\end{equation}
\cite[Proposition 2.2(i)]{ABR-9}. But, via \eqref{52}, we then have for each $n \in \N $ that
$$
\frac{e^{\al_n / l_1}}{n} = v_1 (1) \cdot  \frac{v_{l_1}(n)}{n v_1 (n)}  \le v_1 (1) \cdot  \frac{v_{l_1}(n)}{n} \sum^n_{m=1}  \frac{1}{v_1 (m)}  \le
M v_1 (1)  .
$$
This contradicts \eqref{51} for  $l := l_1$. Hence, $\sC $ does \textit{not} act in $k_0 (V)$.
\end{proof}

\begin{example} \label{ex52} \rm
Define $\al_n := \log (n+1)$ for $n \in \N $. Since $\lim_{n \to \infty } \frac{\log (n)}{\al_n} = 1 \ne 0 $ , the space $k_0 (V)$ is \textit{not} nuclear.
To see that $\sC \in \cL (k_0 (V))$ fix any $k \in \N $ and set $l := k+1$. Noting that $v_r (n) = (n+1)^{1/r}$ for $r, n \in \N $, it follows that
\begin{equation}\label{53}
\frac{v_l (n)}{n} \sum^n_{m=1} \frac{1}{v_k (m)}    = \frac{(n+1)^{1/l}}{n} \sum^n_{m=1}  \frac{1}{ (m+1)^{1/k}}  \le
\frac{2(n+1)^{1/l}}{(n+1)} \sum^{n+1}_{m=1} \frac{1}{ m^{1/k}}  ,
\end{equation}
for each $n \in \N $. If $k =1$, then $l=2$ and it follows from \eqref{53} and the inequality $\sum^{n+1}_{m=1} \frac 1 m \le 1 + \log (n+1)$
that the left-side of \eqref{53} is at most $ \frac{2 (1+ \log (n+1))}{(n+1)^{1/2}}$, for $n \in \N $. For $k > 1 $, using the inequality
$\sum^{n+1}_{m=1} \frac {1}{ m^\delta } \le 1 + \frac{(n+1)^{1- \delta }}{1- \delta }$, $n \in \N $ (valid for each $\delta \in (0 ,1)$), with
$\delta := \frac 1 k$ it follows from \eqref{53} (with $l = k+1$) that
$$
\frac{v_l (n)}{n} \sum^n_{m=1} \frac{1}{v_k (m)}  \le (n+1)^{(\frac{1}{k+1}-1)} + \frac{k (n+1)^{\frac{1}{k+1}-\frac 1 k}}{(k-1)}, \qquad
n \in \N .
$$
In both the cases (i.e., $k=1$ and $k >1$) we see that $\sup_{n \in \N } \frac{v_l (n)}{n} \sum^n_{m=1} \frac{1}{v_k (m)}  < \infty $ and so
$\sC : c_0 (v_k ) \to c_0 (v_l)$ is continuous, \cite[Proposition 2.2(i)]{ABR-9}. Since this is valid for \textit{every} $k \in \N $ and with
$l := k+1$, it follows that $\sC \in \cL (k_0 (V))$.
\end{example}

\begin{example} \label{ex53}\rm
Let $ (j(k))_{k \in \N } \su \N $ be the sequence given by $j (1) := 1$ and $j (k+1) := 2 (k+1) (j(k))^k$, for $k \ge 1 $. Observe that
$j (k+1) > k (j(k))^k + 1 > j (k)$ for all $k \in \N $. Define $\beta = (\beta_n)_{n \in \N }$ via $\beta_n := k (j(k))^k$ for $n = j (k), \ldots,
j (k+1)-1$. Then $\beta $ is non-decreasing with $\lim_{n \to \infty } \beta_n = \infty $. Let $\gamma = (\g_n)_{n \in \N }$ be any
strictly  increasing sequence satisfying $2< \g_n \uparrow 3$. Then the sequence $\al_n := \log (\beta_n + \g_n)$, for $n \in \N $,
satisfies $1 < \al_n \uparrow \infty $ and $\lim_{n \to \infty } \frac{\log (n)}{n} \ne 0 $,  \cite[Remark 2.17]{ASM}. In particular, $k_0 (V)$
it \textit{not} nuclear. To establish that $\sC$ does \textit{not} act in $k_0 (V)$ is suffices to show, for $k:= 1$, that
\begin{equation}\label{54}
\sup_{n \in \N }\frac{v_l (n)}{n} \sum^n_{m=1} \frac{1}{v_1 (m)}  = \infty, \qquad \forall l \in \N .
\end{equation}
So, fix any $l \in \N $. Select $n = j (k)$, for any $k \in \N $, and observe (for \textit{this} $n$) that
\begin{eqnarray*}
&& \frac{v_l (n)}{n} \sum^n_{m=1} \frac{1}{v_1 (m)}   = \frac{(\beta_{j (k)} + \g_{j (k)})^{1/l}}{j (k)} \sum^{j (k)}_{m=1} \frac{1}{\beta_m + \g_m}
\ge  \frac{(\beta_{j (k)} + \g_{j (k)})^{1/l}}{j (k)} \cdot \frac{1}{(\beta_1 + \g_1)} \\
&& \qquad \ge \frac{(k (j (k))^k + \g_{j (k)})^{1/l}}{4 j (k)} \ge \frac{k^{1/l} (j (k))^{(\frac k l)-1}}{4} \ge \frac{k^{1/l} k^{(\frac k l )-1}}{4} ,
\end{eqnarray*}
where we have used $\frac{1}{\beta_1 + \g_1} > \frac 1 4$ and $j (k) \ge k$. Accordingly,
$$
\sup_{n \in \N }\frac{v_l (n)}{n} \sum^n_{m=1} \frac{1}{v_1 (m)}  \ge \sup_{k \in \N }\frac{v_l (j (k))}{j(k)} \sum^{j(k)}_{m=1} \frac{1}{v_1 (m)}
\ge \sup_{k \in \N} \frac{k^{1/l} k^{(\frac k l)-1}}{4} = \infty .
$$
So, \eqref{54} is satisfied and hence, $\sC$ does not act in $k_0 (V)$.

The final two (abstract) results are recorded here in order not to disturb the flow of the text in earlier sections (where these results are
needed). We begin with a fact which is surely known; a proof is included for the sake of self containment.
\end{example}

\begin{lemma}\label{Lereg} Let $E=\ind_k (E_k,\|\ \|_k)$ be a regular  inductive limit of Banach spaces.
Then a subset $\cH\su \cL(E)$ is equicontinuous if and only if the following condition is satisfied:
for every $k \in \N $ there exists $l \in \N $ with $l \ge k $ and $D > 0 $ such that
\begin{equation}\label{eq.Eq}
 \|Tx\|_l\leq D \|x\|_k, \qquad   \forall T\in \cH , \  x\in E_k  .
\end{equation}
\end{lemma}

\begin{proof}
  First,  assume that  $\cH$ is equicontinuous.  Fix $k\in\N$, in which case the  closed unit ball $B_k$ of $E_k$ is bounded in $E$.
The claim is that $C:=\cup_{T\in \cH}T(B_k)$ is bounded in $E$. Indeed, by equicontinuity of $\cH $, given any $0$-neighbourhood $V$ in $E$
there exists a $0$-neighbourhood $U$ in $E$ such that $T(U)\su V$ for all $T\in\cH$. Since $B_k$ is bounded in $E$, there exists $\lambda>0$ such that $B_k\su\lambda U$ and hence, $T(B_k)\su \lambda T(U)\su\lambda V$ for all $T\in \cH$. It follows that $C\su \lambda V$. Since $V$ is arbitrary, it follows that $C$ is bounded in $E$.
But,  $E$ is regular and so  there exists  $l \ge k$ such that $C$ is contained and bounded in $E_l$. Thus, there exists $D>0$ such that $\|Tx\|_l\leq D$
for all $x\in B_k$ and $T\in \cH$.  Accordingly, the stated  condition \eqref{eq.Eq} is satisfied.

Assume that the stated condition \eqref{eq.Eq} is satisfied. Since $E$ is barrelled,   the Banach-Steinhaus principle is available and so it suffices
 to  show that the set $\{Ty\colon T\in \cH\}$ is bounded in $E$ for each  $y\in E$. So, fix  $y\in E$ in which case   $y\in E_k$ for some $k\in\N$.
Selecting $l\ge k$ and $D>0$ according to condition  \eqref{eq.Eq}, we have $\|Ty\|_l\leq D\|y\|_k$ for all $T\in\cH$.
  Hence, the set $\{Ty\colon T\in \cH\}$ is bounded in $E_l$ and so, also in $E$.
\end{proof}

The following result occurs in  \cite[Lemma 5.2]{ABR}.

\begin{lemma}\label{Le-spectrumNV} Let $E=\ind_n (E_n,\| \boldsymbol{\cdot } \|_n)$ be a Hausdorff inductive limit of Banach spaces. Let $T\in \cL(E)$ satisfy the following condition:
\begin{itemize}
\item[\rm (A)] For each $n\in\N$ the restriction $T_n$ of $T$ to $E_n$ maps $E_n$ into itself and belongs to $\cL(E_n)$.
\end{itemize}
Then the following properties are satisfied.
\begin{itemize}
\item[\rm (i)] $\sigma_{pt}(T ; E)=\cup_{n\in\N}\sigma_{pt}(T_n ; E_n)$.
\item[\rm (ii)] $\sigma(T  ; E)\su \cap_{m\in\N} (\cup_{n=m}^\infty\sigma(T_n ; E_n))$. Moreover, if $\lambda \in \cap_{n=m}^\infty \rho(T_n ;  E_n)$ for some $m\in\N$, then $R(\lambda,T_n)$ coincides with the restriction of $R(\lambda,T)$ to $E_n$ for each $n \geq m$.
\item[\rm (iii)] If  $\cup_{n=m}^\infty\sigma(T_n ; E_n)\su \ov{\sigma(T ; E)}$ for some $m\in\N$, then $\sigma^*(T ; E)=\ov{\sigma(T ;E)}$.
\end{itemize}
\end{lemma}

\textbf{Acknowledgements.}
The research of the first two authors was partially
supported by the projects  MTM2016-76647-P (Spain). The second author gratefully acknowledges the support of the Alexander von Humboldt Foundation.

\bigskip
\bibliographystyle{plain}

\end{document}